\documentclass[11pt, reqno]{amsart}

\usepackage{amssymb, amscd, latexsym}
\usepackage{amsmath}
\usepackage{amsthm}
\usepackage{graphicx}
\usepackage[all]{xy}
\usepackage{float}
\usepackage{mathrsfs,centernot}
\usepackage{pgf,tikz}
\usetikzlibrary{arrows}
\usepackage{array}
\usepackage{arydshln}
\usepackage{hyperref}

\hypersetup{
	pdftitle={Research Paper},
	pdfauthor={Mina Bigdeli},
	pdfsubject={linear resolution of ideals},
	pdfcreator={Mina Bigdeli},
	colorlinks,
	linkcolor=blue,
	citecolor=magenta,
	anchorcolor=red,
	bookmarksopen,
	urlcolor=red,
	filecolor=red,
	pdfpagetransition={Wipe}
}

\theoremstyle{plain}
\newtheorem{thm}{Theorem}[section]
\newtheorem{cor}[thm]{Corollary}
\newtheorem{lem}[thm]{Lemma}
\newtheorem{remno}[thm]{Notation and Observation}
\newtheorem{prop}[thm]{Proposition}

\newtheorem{rem}[thm]{Remark}
\theoremstyle{definition}

\newtheorem{ex}[thm]{Example}

\def\implies{\ifmmode\Rightarrow \else
	\unskip${}\Rightarrow{}$\ignorespaces\fi}

\def\Index{\mathrm{index}}
\def\projdim{\mathrm{pd}}
\def\im{\mathrm{Im\ }}
\def\supp{\mathrm{supp}}
\newcommand{\x}[1]{\mathbf{x}_{#1}}
\newcommand{\e}[1]{\mathbf{e}_{#1}}
	\definecolor{darkblue}{rgb}{0.0, 0.0, 0.55}
	
\begin{document}
	
	\title[Edge ideals with almost maximal finite index]{Edge ideals with almost maximal finite index and their powers} 
	\author[M.~Bigdeli]{Mina Bigdeli}
	\address{School of Mathematics\\ Institute for Research in Fundamental Sciences (IPM)\\  P.O.Box: 19395-5746\\ Tehran, Iran}
	\email{mina.bigdeli98@gmail.com, mina.bigdeli@ipm.ir}
	
	
	\begin{abstract}
		A graded ideal $I$ in  $\mathbb{K}[x_1,\ldots,x_n]$, where $\mathbb{K}$ is a field,  is said to have almost maximal finite index  if its minimal free resolution is linear  up to the  homological degree $\projdim(I)-2$, while it is  not linear at the homological degree $\projdim(I)-1$, where $\projdim(I)$ denotes the projective dimension of $I$. In this paper we  classify  the graphs whose edge ideals have this property.  This in particular shows  that for edge ideals  the property of having almost maximal finite index does not depend on the characteristic of $\mathbb{K}$. We also compute the non-linear Betti numbers of these ideals. Finally, we show that  for the edge ideal $I$ of a graph $G$ with almost maximal finite index,  the ideal $I^s$ has a linear resolution for $s\geq 2$ if and only if the complementary graph $\bar{G}$ does not contain induced cycles of length $4$.
	\end{abstract}
	\dedicatory{Dedicated to J\"urgen Herzog,  on the occasion
     of his $80$th birthday}
	\maketitle

	\section*{Introduction}
	In this paper, we consider the edge ideals whose minimal free resolution has relatively large number of linear steps. Let $I$ be a graded ideal in the polynomial ring $S=\mathbb{K}[x_1,\ldots,x_n]$, where $\mathbb{K}$ is  a field, generated by homogeneous polynomials of degree $d$. The ideal is called $r$-steps linear, if $I$ has a linear resolution up to the homological degree $r$, that is the graded Betti numbers $\beta_{i,i+j}(I)$ vanish  for all $i\leq r$ and all $j>d$. The number
	\[
	\Index(I)=\inf\{r:\  \text{$I$ is not $r$-steps linear}\}
	\]
	is called the  Green--Lazarsfeld index (or briefly index) of $I$. A related invariant, called the $N_{d,r}$-property, was first considered by Green and Lazarsfeld in \cite{GL1, GL2}.  In the paper \cite{BC} the  index was introduced for the quotient ring $S/I$, where $I$ is generated by quadratics,  to be the largest integer $r$ such that the $N_{2,r}$-property holds. It is in general  very hard to determine the value of the index. One reason is that this value, in general, depends on the characteristic of $\mathbb{K}$. The index of quadratic monomial ideals is more studied in the literature taking  advantage of  some combinatorial methods. Indeed, since the index  is preserved passing through polarization, one may reduce to the case of squarefree quadratic monomial ideals which can be viewed as the edge ideals of simple graphs, and the index of these ideals is proved to be characteristic independent, see \cite[Theorem~2.1]{EGHP}. 

					The main question  regarding the study of the index of edge ideals is to classify the graphs with respect to the index of their  edge ideals, in particular, it is more interesting to see when the   index attains its largest or smallest value. 
		In 1990,  Fr\"oberg \cite{Fr} classified the graphs whose edge ideals  have  a linear resolution. A graded ideal $I$ is said to have a linear resolution if $\Index(I)=\infty$. In fact Fr\"oberg showed that given a graph $G$, its edge ideal $I(G)$ has a linear resolution over all fields if and only if the complement $\bar{G}$ of $G$ is chordal, which means that    all cycles in $\bar{G}$ of length $>3$ have a chord. In 2005, Eisenbud et al. \cite{EGHP} gave a purely combinatorial description of the index of edge ideals in terms of the size of the smallest  cycle(s) of length $>3$ in the complementary graph, c.f. Theorem~\ref{index of graphs}. This result shows that the index gets its smallest value $1$ if and only if $G$ admits a gap, i.e. $\bar{G}$ contains an  induced cycle of length $4$. If the index of $I$  attains the largest  finite value, we have $\Index(I)=\projdim(I)$, where $\projdim(I)$ denotes the projective dimension of $I$. In this case the  ideal $I$ is said to have maximal finite index, see \cite{BHZ}. In  \cite[Theorem~4.1]{BHZ}, it was shown that the edge ideal $I(G)$  has maximal finite index if and only if $\bar{G}$ is  a cycle of length~$>3$. In this paper, we proceed one more step and consider the edge ideals $I(G)$ with $\Index(I(G))=\projdim(I(G))-1$. We  call them edge ideals with almost maximal finite index. In Section~\ref{classify}  of this paper we  precisely determine the simple graphs whose edge ideals have this property, see Theorem~\ref{check-out}. These graphs are presented in Figures~\ref{type a}--\ref{type d}. In particular, it is deduced that the property of having almost maximal finite index is characteristic independent for edge ideals, though this is not the case for ideals generated in higher degrees, as discussed in the beginning of Section~\ref{classify}. It is also seen that the graded Betti numbers of these edge ideals do not depend on the characteristic of the base field. 		We will  compute the   Betti numbers in the non-linear strands in Proposition~\ref{Bettis of almost}. 	The main tool used throughout this section   is  Hochster's formula, Formula~(\ref{Hochster}).

In the second half of the paper we study the index of  powers of edge ideals with almost maximal finite index.	 Although, for arbitrary ideals, many properties  such as depth, projective dimension or regularity stabilize for large powers (see e.g., \cite{Ba,Ch,Ca,Co, CHT,HH,  HHZh1,HHZ1,HW}), their initial behaviour is often quite mysterious. However, edge ideals behave more controllable from the beginning. In the study of the  index of powers of edge ideals, one of the main  results is due to Herzog, Hibi and Zheng  \cite[Theorem~3.2]{HHZh1}. They  showed that for a graph $G$, all  powers of the edge ideal $I(G)$ have a linear resolution if and only if so does $I(G)$. On the other hand, it was shown in \cite[Theorem~3.1]{BHZ} that  all  powers of $I(G)$ have  index $1$ if and only if  $I(G)$ has also index $1$. In the same paper it was proved that if $I(G)$ has maximal finite index~$>1$, then  $I(G)^s$ has a linear resolution for all $s\geq 2$. This shows that chordality of the complement of $G$ is not a necessary condition on $G$ so that all high powers of its edge ideal have a linear resolution. Francisco, H\`a and Van~Tuyl proved, in a personal communication, that being gap-free is a necessary condition for a graph $G$  in order that a power of its edge ideal has a linear resolution (see also \cite[Proposition~1.8]{NP}).  However, Nevo and Peeva showed, by an example,  that  being gap-free alone  is not  a sufficient condition  so that all high powers of the edge ideal have a linear resolution \cite[Counterexample~1.10]{NP}. Later, Banerjee \cite{Ba}, and Erey \cite{Er, Er1} respectively proved that  if a gap-free  graph $G$ is also cricket-free or diamond-free or $C_4$-free, then the  ideal $I(G)^s$ has a linear resolution for all $s\geq 2$.  The definition of these concepts are recalled in Section~\ref{powers of almost maximal}.  

	Section~\ref{powers of almost maximal} is devoted to answer the question whether the high powers of edge ideals with almost maximal finite index have a linear resolution. Not all graphs whose edge ideals have this property are cricket-free or diamond-free.	However,  using some formulas for an upper bound of the regularity of either powers of edge ideals or in general monomial ideals offered in \cite[Theorem~5.2]{Ba},  and \cite[Lemma~2.10]{DHS} respectively,  we give a positive answer to this question in case the graphs are gap-free, see Theorem~\ref{powers}. We will prove this theorem in several parts, mainly in Theorem~\ref{main G_a3} and Theorem~\ref{I^k has lin res}.
	
	 	   Theorem~\ref{powers} together with \cite[Theorem~4.1]{BHZ} yield the following consequence which is a partial  generalization of  the result of Herzog et al. in \cite[Theorem~3.2]{HHZh1}.
	 	  \begin{thm}\label{bound}
	 	  	Let $G$ be a simple gap-free graph and let $I\subset S$ be its edge ideal. Suppose $\projdim(I)\!-\!\Index(I)\leq\!1$. Then $I^s$ has a linear resolution over all fields for any $s\geq 2$.
	 	  \end{thm}
	 	One may ask which is the largest integer $c$ such that Theorem~\ref{bound} remains valid  if one replaces $\projdim(I) - \Index(I) \leq 1$ by $\projdim(I) - \Index(I) \leq c$. Computation by {\em Macaulay~2}, \cite{M2}, shows that in the example of Nevo and Peeva \cite[Counterexample~1.10]{NP},  $\Index(I)=2$,  and $\projdim(I)=8$. Hence $c$ must be an integer with $1\leq c\leq 5$. 
	 
	 	\subsection*{Acknowledgement} Research was supported by a grant from IPM. This work was initiated   while the  author was resident at MSRI during the Spring 2017 semester and supported by National Science Foundation under Grant No. DMS-1440140. Theorem~\ref{check-out} is a consequence of a question David Eisenbud asked the author.  She would like to thank him for the invaluable discussions  throughout her postdoctoral fellowship at MSRI. She also extends her gratitude to Rashid Zaare-Nahandi for his comments on this manuscript.
		
		Finally, the author would like to express her  appreciation to the anonymous referee for the remarkable comments and useful suggestions which helped to improve the manuscript. 
	
		\section{Preliminaries}\label{section 1}
			In this section we recall some concepts, definitions and results from Commutative Algebra and Combinatorics which will be used throughout the paper. Let $S=\mathbb{K}[x_1, \ldots, x_n]$ be the polynomial ring over a field $\mathbb{K}$ with $n$ variables, and let $M$ be a finitely generated graded $S$-module. Let the sequence 
		$$
		0\to  F_p\to\cdots \to F_2 \to F_1 \to F_0 \to M \to 0
		$$
		be the minimal graded free  resolution of $M$, where for  all $i \geq 0$ the modules $F_i = \oplus_j S(-j)^{\beta_{i,j}^{\mathbb{K}}(M)}$ are free $S$-modules of rank $\beta_{i}^{\mathbb{K}}(M):=\sum_{j}\beta_{i,j}^{\mathbb{K}}(M)$.
		The numbers $\beta_{i,j}^{\mathbb{K}}(M) = \dim_{\mathbb{K}} \mbox{Tor}^S_i(M, \mathbb{K})_j$ are called the \textit{graded Betti numbers} of $M$ and $\beta_i^{\mathbb{K}}(M)$ is called the $i$-th {\em Betti number} of $M$. We write $\beta_{i,j}(M)$ for $\beta_{i,j}^{\mathbb{K}}(M)$ when the field is fixed. 
		The {\em projective dimension} of $M$, denoted by $\projdim(M)$, is the largest $i$ for which $\beta_{i}(M)\neq 0$. The  {\em Castelnuovo-Mumford regularity } of $M$, $\mathrm{reg}(M)$, is defined to be $$\mathrm{reg}(M)=\sup\{j-i:\ \beta_{i,j}(M)\neq 0\}.$$ 
	
		
		Let $I$ be a graded ideal of $S$ generated in a single degree $d$. 
		The {\em Green--Lazarsfeld index} (briefly index) of $I$, denoted by $\Index(I)$, is defined to be 
		$$\Index(I)=\inf\{i:\ \beta_{i,j}(I)\neq 0,\ \text{for some } j>i+d\}.$$
		Since $\beta_{0,j}(I)=0$ for all $j>d$, one always has $\Index(I)\geq 1$. 
			  The ideal $I$ is said to have a {\em $d$-linear resolution} if $\Index(I)=\infty$. This means that for all $i$,  $\beta_{i}(I)=\beta_{i,i+d}(I)$, and this is the case if and only if $\mathrm{reg}(I)=d$. Otherwise $\Index(I)\leq \projdim(I)$. In case $I$ has the largest possible finite index, that is $\Index(I)=\projdim(I)$, $I$ is said to have {\em maximal finite index}. 
		
				\medspace
		In Section~\ref{classify} of this paper we deal with squarefree monomial ideals generated in degree $2$. These ideals are the edge ideals of simple graphs.  Recall that  a {\em simple} graph is a graph with no loops and no multiple edges, and given a graph $G$ on the vertex set $[n]:=\{1,\ldots,n\}$, its edge ideal $I(G)\subset S$ is an ideal generated by all quadratics $x_ix_j$, where $\{i,j\}$ is an edge in $G$. We denote by $E(G)$ the set of all edges of $G$, and by $V(G)$ the vertex set of $G$. For a vertex $v\in V(G)$, the neighbourhood $N_G(v)$ of $v$ in $G$  is defined to be 
		$$N_G(v)=\{u\in V(G):\  \{u,v\}\in E(G) \}.$$
		The  complement $\bar{G}$ of $G$ is a graph on $V(G)$ whose edges are those pairs of $V(G)$ which do not belong to $E(G)$. 	The simplicial  complex 
		 $$\Delta(G)=\{F\subseteq V(G):\ \text{for all } \{i,j\}\subseteq F \text{ one has } \{i,j\} \in E(G)\}$$
		  is called the {\em flag complex} of $G$. The {\em{independence complex}} of $G$ is the flag complex of $\bar{G}$.	One can check that $I(G)=I_{\Delta(\bar{G})}$, where $I_{\Delta(\bar{G})}$ is the Stanley-Reisner ideal of $\Delta(\bar{G})$. 
		 We assume that the reader is familiar with the  definition and elementary properties of simplicial complexes.  For more details  consult with \cite{HHBook}.

		 \medspace
		 The main tool  used widely in Section~\ref{classify} for the computation of the graded Betti numbers  is Hochster's formula~\cite[Theorem~8.1.1]{HHBook}.  Let $\Delta$ be a simplicial complex on $[n]$, and let $\tilde{C}(\Delta, \mathbb{K})$ be the augmented oriented chain complex of $\Delta$ over a field $\mathbb{K}$ with the  differentials 
		 						\begin{align*}
&\quad\quad\quad\quad\quad\quad	\partial_i: \bigoplus_{F\in\Delta\atop \dim F=i}\mathbb{K}F\to \bigoplus_{G\in\Delta\atop \dim G=i-1}\mathbb{K}G,\\ 
&\partial_i([v_0,\ldots,v_{i}])=\sum_{0\leq j\leq i}(-1)^{j}[v_0,v_1,\ldots, v_{j-1},v_{j+1}, \ldots, v_{i}],
						\end{align*}
						where by $[v_0,v_1,\ldots,v_{i}]$ we mean the face $\{v_0,v_1,\ldots,v_{i}\}\subseteq [n]$ of $\Delta$ with $v_0<v_1<\cdots<v_i$.
		 		 Hochster's formula states that for the Stanley-Reisner ideal $I:=I_{\Delta}\subset S$  one has
		 \begin{eqnarray}\label{Hochster}
		 \beta_{i,j}(I)=\sum_{W\subseteq [n],\ |W|=j} \dim_\mathbb{K} \widetilde{H}_{j-i-2}(\Delta_W;\mathbb{K}),
		 \end{eqnarray}
		 where $\Delta_W$ is the induced subcomplex of $\Delta$ on $W$ and $\widetilde{H}_i(\Delta_W;\mathbb{K})$ is the $i$-th reduced homology of the complex $\widetilde{C}(\Delta_W, \mathbb{K})$. We denote by $\partial_i^W$ the differentials of the chain complex $\widetilde{C}(\Delta_W, \mathbb{K})$.

		 \medspace
		 Theorem~\ref{index of graphs} which is due to  Eisenbud et al. \cite{EGHP}  provides a combinatorial method for determining the index of the edge ideal of a graph.  To this end, one needs to consider the length of the minimal cycles of the complementary graph. A minimal cycle is  an induced cycle of length$>3$, and by an induced cycle we mean a cycle with no chord. The length of an induced cycle $C$ is denoted by $|C|$.
		\begin{thm}[{\cite[Theorem~2.1]{EGHP}}]\label{index of graphs}
			Let $I(G)$ be the edge ideal of a simple graph $G$. Then 
			$$\Index(I(G))=\inf\{|C|: \ C \text{ is a minimal cycle} \text{ in } \bar{G}\}-3.$$
		\end{thm}


\section{Edge ideals with almost maximal final index}\label{classify}

A graded ideal $I\subset S$ is said to have {\em almost maximal finite index} over $\mathbb{K}$ if $\Index(I)=\projdim(I)-1$. Since, in general, $\projdim(I)$ and $\Index(I)$ depend on the characteristic of the base field,  the property of having almost maximal finite index may also be characteristic dependent.  For example, setting $\Delta$ to be  a triangulation of a real projective plane,  the  Stanley-Reisner ideal of $\Delta$  is generated in degree $3$ and it has almost maximal finite index over all fields of characteristic $2$, while it has a linear resolution over  other fields (cf. \cite[\S 5.3]{BHBook}). However, as we will see in Corollary~\ref{final note}, in the case of quadratic monomial ideals, having almost maximal finite index  is characteristic independent. Note that, although by Theorem~\ref{index of graphs}, the index of an arbitrary edge ideal does not depend on the base field, its projective dimension may depend.  M.~Katzman presents a graph in \cite[Section~4]{Ka}  whose edge ideal has different projective dimensions  over different fields.

	\medspace
In  this section, we  give a classification of  the graphs whose edge ideals have almost maximal finite index.   
We will present this classification in Theorem~\ref{check-out}, but before, we need some intermediate steps which give more insight about the complement of such graphs. 

Unless otherwise stated, throughout this section, $G$ is a simple graph on the vertex set $[n]$ and $\bar{G}$ is its complement, $\Delta$ denotes the independence complex $\Delta(\bar{G})$, and $\partial$, $\partial^W$ denote respectively the differentials of the augmented oriented chain complexes of $\Delta(\bar{G})$, $\Delta(\bar{G})_W$ over a fixed field $\mathbb{K}$. 

\medspace
First, in order to avoid  repetition of some arguments, we gather some facts   which will be used frequently in the sequel in the following Observation. Meanwhile, we also fix some notation.
\begin{remno}\label{connected induced graphs}\hfill\par\rm
	Let $G$ be a simple graph on the vertex set $[n]$ and let $I:=I(G)\subset S$ be its edge ideal.
		
	{\bf (O-1)} The graph $G$ is connected if and only if its flag complex $\Delta(G)$ is connected. On the other hand 	for an arbitrary simplicial complex $\Gamma$ and any field $\mathbb{K}$, $$\dim_\mathbb{K}\widetilde{H}_0(\Gamma;\mathbb{K})=(\text{Number of connected components of }\Gamma)-1,$$  see  \cite[Problem~8.2]{HHBook}. 
		Moreover, 		for any subset $W\subseteq [n]$ one has 	$\Delta(G_W)=\Delta(G)_W$, where $G_W$ is the induced subgraph of $G$ on the vertex set $W$.
		It follows that $G_W$ is connected if and only if $\widetilde{H}_0(\Delta(G)_W;\mathbb{K})=0$. Now if   $\beta_{i,i+2}(I)=0$ for some $i$, then by Hochster's formula 
		$\widetilde{H}_{0}(\Delta_W;\mathbb{K})=0$  and hence $\bar{G}_W$ is connected for all $W\subseteq [n]$ with $|W|=i+2$. 
		
		\medskip
	{\bf (O-2)} Throughout, by $P=u_1-u_2-\cdots-u_r$ in  $G$ we mean a path in $G$ on $r$ distinct vertices with the set of edges $\bigcup_{1\leq i\leq r-1}\{\{u_i,u_{i+1}\}\}$. If, in addition $\{u_1,u_r\}\in E(G)$, then 
		$C=u_1-u_2-\cdots-u_{r}-u_1$ is a cycle  in  $G$.  Then   
		\begin{align}\label{cycle kernel}
		T(C):=(\sum_{1\leq i\leq r-1}[u_i,u_{i+1}])-[u_1,u_{r}]\in \ker \partial^{\Delta(G)}_1,
		\end{align}
		where $\partial^{\Delta(G)}$ denotes the differentials  of the  chain complex of $\Delta(G)$. It is shown in \cite[Theorem~3.2]{Co} that $\widetilde{H}_1(\Delta(G); \mathbb{K})\neq 0$ if and only if there exists a minimal cycle $C$ in $G$ such that $T(C)\notin \im \partial_2^{\Delta(G)}$. Indeed, it is proved that  $\widetilde{H}_1(\Delta(G); \mathbb{K})$ is minimally generated by the nonzero homology classes $T(C)+\im \partial_2^{\Delta(G)}$, where $C$ is a minimal cycle in $G$.
		
		If $C$ is the base of a cone whose apex is  the vertex $u_{r+1}$, then 
		$$T(C)=\partial_{2}^{\Delta(G)}((\sum_{1\leq i\leq r-1}[u_{r+1},u_i,u_{i+1}])-[u_{r+1},u_1,u_{r}])$$
		which implies that $T(C)+\im\partial_2^{\Delta(G)}=0$. Recall that  an $r$-gonal {\em cone} with the apex $a$ is a graph $G'$ with the vertex set $V(G')=V(C)\cup\{a\}$, where $a\notin V(C)$ and   $C$ is an $r$-cycle in $G'$ which is called the base of $G'$,   and    $E(G')=E(C)\cup\{\{a,u_i\}: u_i\in V(C)\}$. 
				
		\medskip
	{\bf (O-3)} Now let $D$ be an $r$-gonal dipyramid in $G$; that is a subgraph of $G$ with the vertex set   $V(D)=V(C)\cup\{a,b\}$ and $E(D)= \bigcup_{1\leq i\leq r}\left(\{a,u_i\}\cup\{b,u_i\}\right) \cup E(C)$ where  $C$ is an $r$-cycle as above which is called the {\em waist } of $D$. Then
			\begin{align}\label{dipyramid kernel}
		T(D):=(\sum_{1\leq i\leq r-1}[a,u_i,u_{i+1}]-[b,u_i,u_{i+1}])-[a,u_1,u_{r}]+[b,u_1,u_{r}]\in \ker \partial_2^{\Delta(G)}.
		\end{align}
	\par	{\bf (O-4)} Suppose $\Index(I)=t$.  By Theorem~\ref{index of graphs},  $\bar{G}$ contains a minimal cycle $C=u_1-u_2-\cdots-u_{t+3}-u_1$ which has the smallest length among all minimal cycles of $\bar{G}$.  
	\begin{itemize}
	\item[$(i)$]  If $\beta_{t+1,t+4}(I)=0$, then $\widetilde{H}_{1}(\Delta_W;\mathbb{K})=0$ for all $W\subseteq [n]$ with $|W|=t+4$.  
		Set $W=\{u_{t+4}\}\cup V(C)$ for an arbitrary vertex $u_{t+4}\in [n]\setminus V(C)$. Then $C$ is a minimal cycle in  $\bar{G}_W$ and   $T(C)\in\ker\partial_1^{W}$ implies that $T(C)\in\im\partial_2^{W}$. 
		 It follows that  each edge $e$ of $C$ is contained in a $2$-face $F_e$ of ${\Delta_W}$. Since $C$ is minimal, we must have $F_e=e\cup\{u_{t+4}\}$ which means that $u_{t+4}$ is adjacent to all vertices of $C$ in $\bar{G}$ and hence $\bar{G}_W$ is a cone.
		\item[$(ii)$] 	 If $\beta_{t+2,t+5}(I)=0$, then $\widetilde{H}_{1}(\Delta_W;\mathbb{K})=0$ for all $W\subseteq [n]$ with $|W|=t+5$. 	Set $W=\{u_{t+4}, u_{t+5}\}\cup V(C)$ for  arbitrary vertices $u_{t+4}, u_{t+5}\in [n]\setminus V(C)$. 
		As in (i), $T(C)=\partial_2^{W}(L)$ for some $L\in \bigoplus_{F\in \Delta_W\atop{\dim F=2}} \mathbb{K}F$, and hence   each edge of $C$ is  contained in a $2$-face of $\Delta_W$. 
		It follows that for each edge $e$ of $C$ either $\{u_{t+4}\}\cup e\in\Delta_W$ or $\{u_{t+5}\}\cup e\in \Delta_W$. If  for all $e\in E(C)$ one has  $\{u_{t+4}\}\cup e\in\Delta_W$, then $\Delta_W$ contains a cone. Same holds if we replace $u_{t+4}$ with $u_{t+5}$.  Suppose $\{u_{t+4}\}\cup e, \{u_{t+5}\}\cup e' \notin\Delta_W$ for some $e, e'\in E(C)$, which implies that $u_{t+4}, u_{t+5}$ are not adjacent to all vertices of $C$ in $\bar{G}$. 
			Without loss of generality suppose $\{u_{t+5},u_1,u_2\}\notin \Delta_W$. It follows that $\{u_{t+4}\}\cup\{u_1,u_2\}\in \Delta_W$. If $\{u_1,u_2\}$ is the only edge $e$ of $C$ with $\{u_{t+4}\}\cup e\in\Delta_W$, then for all $e'\in E(C)$ with $e'\neq \{u_1,u_2\}$ one has $\{u_{t+5}\}\cup e'\in  \Delta_W$. In particular, $\{u_1,u_{t+3}, u_{t+5}\},\{u_2,u_3,u_{t+5}\}\in \Delta_W$ which  implies by the definition of $\Delta_W=\Delta(\bar{G}_W)$ that $\{u_{t+5},u_1,u_2\}\in \Delta_W$, a contradiction. Since 
 $u_{t+4}$ is not adjacent to all vertices of $C$ in $\bar{G}$, and since $\{u_{t+4},u_1\},\{u_{t+4},u_2\}\in E(\bar{G})$, it follows that  there exists  $3\leq j\leq t+3$ such that $\{u_j, u_{t+4}\}\notin E(\bar{G})$. Let $a, b $ be respectively  the biggest and the smallest  integers  with $2\leq a<j<b\leq t+3$ for which $\{u_a,u_{t+4}\},\{u_b,u_{t+4}\}\in E(\bar{G})$. If such $b$ does not exist we let $b=1$ which implies that  $a\neq 2$ because otherwise $\{u_{t+5}\} \cup e' \in \Delta_W$ for all $e' \in E(C) \setminus \{\{u_1, u_2\}\}$, so $u_{t+5}$ is adjacent to all vertices of $C$ in $\bar{G}$. 
		Now  if $b\neq 1$, then $C':=u_{t+4}-u_a-u_{a+1}-\cdots-u_{b}-u_{t+4}$ is a minimal cycle in $\bar{G}_W$ of length  $b-a+2$, and if $b=1$  then $C':=u_{t+4}-u_a-u_{a+1}-\cdots-u_{t+3}-u_{1}-u_{t+4}$ is a minimal cycle of length  $t+6-a$. Since $\Index(I)=t$,  we must have $|C'|\geq t+3$ in either case, and so  $\{a,b\}=\{1,3\}$ if $b=1$, and $\{a,b\}=\{2,t+3\}$ if $b\neq 1$. In  both cases  the vertex $u_{t+4}$ is adjacent to only three successive vertices of $C$ in $\bar{G}$.  Without loss of generality we may assume that $u_1,u_2,u_3$ are these three vertices. Thus $u_{t+4}$ is adjacent to only two edges $\{u_1,u_2\}, \{ u_2,u_3\}$ of $C$ in $\bar{G}$ and hence   $\{u_1,u_3,u_4\ldots, u_{t+3}\}\subseteq N_{\bar{G}}(u_{t+5})$. It follows that $\{u_{t+5}, u_2\}\notin E(\bar{G})$ because  $u_{t+5}$ is not adjacent to all vertices of $C$ in $\bar{G}$. Thus we get the minimal $4$-cycle $C'':=u_{t+5}-u_1-u_2-u_3-u_{t+5}$. It follows that $t=1$ because $\Index(I)=t$. Therefore $|C|=4$ and the only $2$-faces of $\Delta_W$ containing an edge of $C$ are  $\{u_1,u_2,u_5\}, \{u_2,u_3,u_5\}, \{u_3,u_4,u_6\}, \{u_1,u_4,u_6\}$. But no linear combination of theses faces will result in $L$ with $\partial_{2}^{W}(L)=T(C)$. We need more $2$-faces in $\Delta_W$. It follows that $\{u_i, u_5,u_6\}\in \Delta_W$ for some $1\leq i\leq 4$. 
			 In particular, $\{u_5,u_6\}\in E(\bar{G})$. This forms a graph $\bar{G}_W$ which is drawn as the  graph $G_{(d)_2}$ in Figure~\ref{type d}.
			\item[$(iii)$ ]If $\beta_{t+2,t+6}(I)=0$, then $\widetilde{H}_{2}(\Delta(\bar{G})_W;\mathbb{K})=0$ for all $W\subseteq [n]$ with $|W|=t+6$. Suppose $\bar{G}$ contains a dipyramid $D$ with the vertex set $\{u_{t+4}, u_{t+5}\}\cup V(C)$, where the waist $C$  is a minimal cycle of length $t+3$. 
		Set $W=\{u_{t+4}, u_{t+5},u_{t+6}\}\cup V(C)$ for  arbitrary vertex $ u_{t+6}\in [n]\setminus (V(C)\cup \{u_{t+4}, u_{t+5}\})$. Then by (O-3) one has $T(D)\in \ker\partial_2^{W}$ and hence  $T(D)\in \im\partial_3^W$. This implies that  each $2$-face of $D$ is contained in a $3$-face of ${\Delta_W}$. Since $C$ is minimal, it follows that  either $\{u_{t+4},u_{t+5}\}\in E(\bar{G})$ or  $\{u_{t+4}, u_{t+5}\}\cup V(C)\subseteq N_{\bar{G}}(u_{t+6})$. 
			\end{itemize}	
			\end{remno} 
				\begin{ex}\label{mesal}
			Here we give $7$ types of the graphs $G$ whose edge ideal $I:=I(G)$ has almost maximal finite index over all fields. Indeed, we present the complementary graphs  $\bar{G}$ for which $\projdim(I)=\Index(I)+1$. Take $t = 1$ for the cases (c) and (d) below. Since the smallest minimal cycles in the  following graphs  $\bar{G}$ are of length $t+3\geq 4$, by Theorem~\ref{index of graphs} we have $\Index(I)=t$. We show that $\projdim(I)=t+1$. Note that as it is also clear from Hochster's formula, $\beta_{i,j}(I)=0$ for all $j<i+2$ and hence, in order to show that $\projdim(I)=t+1$ it is enough  to prove $\beta_{t+1,j}(I)\neq 0$ for some $j\geq t+3$ and $\beta_{t+2,j}(I)= 0$ for all $t+4\leq j\leq n$. 
			The argument below is independent of the choice of the  base field. 
			
			\medspace
			\indent (a) Let $\bar{G}$ be either of the graphs $G_{(a)_1}, G_{(a)_2}, G_{(a)_3}$ shown in Figure~\ref{type a} with $t\geq 1$. The two  graphs $G_{(a)_1}, G_{(a)_2}$ have one minimal cycle $C=1-2-\cdots-(t+3)-1$, and the graph $G_{(a)_3}$, has two minimal cycles $C$ and  $C'=1-(t+4)-3-4-\cdots-(t+3)-1$.  			
			Setting $W=[t+4]$, we have $T(C)\in \ker \partial_1^W$ by (O-2). Since $t>0$, there are edges of $C$ in all three graphs which are not contained in a $2$-face of $\Delta_W$. In particular,   $T(C)\notin \im\partial_2^W$. Hence $\widetilde{H}_1(\Delta_W;\mathbb{K})\neq 0$ which implies that $\beta_{t+1,t+4}(I)\neq 0$. Thus $\projdim(I)\geq t+1$. If $\beta_{t+2,j}(I)\neq 0$ for some $j$, then there exists $W\subseteq [t+4]$ with $|W|=j$ such that $\widetilde{H}_{|W|-t-4}(\Delta_W;\mathbb{K})\neq 0$. It then follows that $W=[t+4]$ and $\widetilde{H}_0(\Delta_W;\mathbb{K})\neq 0$. But    ${\bar{G}}_W=\bar{G}$ is connected  meaning that $\Delta_W$ is connected, by  (O-1). Hence $\widetilde{H}_0(\Delta_W;\mathbb{K})=0$, a contradiction. Therefore $\projdim(I)=t+1$.
			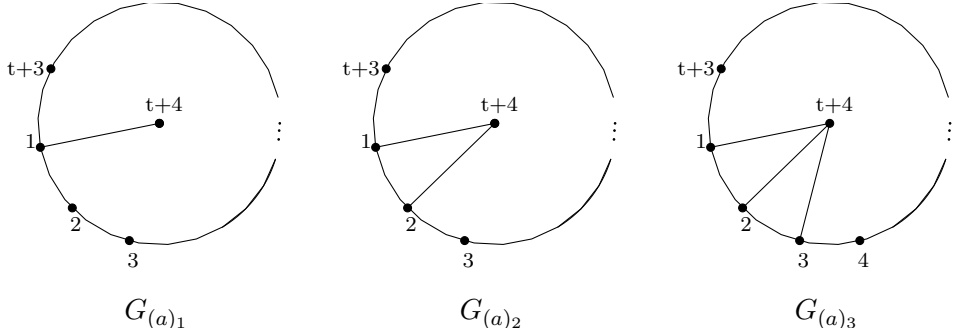
\begin{figure}[ht!]
				\begin{center}
					\hspace*{-2cm}
					\begin{tikzpicture}[line cap=round,line join=round,>=triangle 45,x=0.4cm,y=0.4cm]
					\clip(54.5,3.2) rectangle (68.5,14);
					\draw [color=black] (64.,10.)-- (60.04207935839705,9.202683366710461);
					\draw (59.2,9.98) node[anchor=north west] {\begin{scriptsize}1\end{scriptsize}};
					\draw (62.6,6) node[anchor=north west] {\begin{scriptsize}3\end{scriptsize}};
					\draw (60.7,7.2) node[anchor=north west] {\begin{scriptsize}2\end{scriptsize}};
					\draw (63.2,11.2) node[anchor=north west] {\begin{scriptsize}t+4\end{scriptsize}};
					\draw (58.6,12.4) node[anchor=north west] {\begin{scriptsize}t+3\end{scriptsize}};
					\draw (62.5,4.5) node[anchor=north west] {$G_{(a)_1}$};
					\draw [shift={(64.,10.)},line width=0.4pt]  plot[domain=5.257759132311878:5.982897107046088,variable=\t]({1.*4.037431066773457*cos(\t r)+0.*4.037431066773457*sin(\t r)},{0.*4.037431066773457*cos(\t r)+1.*4.037431066773457*sin(\t r)});
					\draw [shift={(64.,10.)},color=black]  plot[domain=0.21715133778930423:5.982897107046091,variable=\t]({1.*4.037431066773395*cos(\t r)+0.*4.037431066773395*sin(\t r)},{0.*4.037431066773395*cos(\t r)+1.*4.037431066773395*sin(\t r)});
					\begin{scriptsize}
					\draw [fill=black] (64.,10.) circle (1.5pt);
					\draw [fill=black] (64.,10.) circle (1.5pt);
					\draw [fill=black] (60.39320973638185,11.814363142597488) circle (1.5pt);
					\draw [fill=black] (60.04207935839705,9.202683366710461) circle (1.5pt);
					\draw [fill=black] (61.10367540176967,7.187144966296193) circle (1.5pt);
					\draw [fill=black] (68.,10.) circle (0.3pt);
					\draw [fill=black] (67.98,9.76) circle (0.3pt);
					\draw [fill=black] (63,6.1) circle (1.5pt);
					\draw [fill=black] (67.96341651367774,9.465941589917382) circle (0.3pt);
					\end{scriptsize}
					\end{tikzpicture}
					\hspace*{-1.4cm}
					\begin{tikzpicture}[line cap=round,line join=round,>=triangle 45,x=0.4cm,y=0.4cm]
					\clip(54.5,3.2) rectangle (68.5,14);
					\draw [color=black] (64.,10.)-- (60.04207935839705,9.202683366710461);
					\draw [color=black] (64.,10.)-- (61.10367540176967,7.187144966296193);
					\draw (59.2,9.98) node[anchor=north west] {\begin{scriptsize}1\end{scriptsize}};
					\draw (60.7,7.2) node[anchor=north west] {\begin{scriptsize}2\end{scriptsize}};
					\draw (62.6,6) node[anchor=north west] {\begin{scriptsize}3\end{scriptsize}};
					\draw (63.2,11.2) node[anchor=north west] {\begin{scriptsize}t+4\end{scriptsize}};
					\draw (58.6,12.4) node[anchor=north west] {\begin{scriptsize}t+3\end{scriptsize}};
					\draw (62.5,4.5) node[anchor=north west] {$G_{(a)_2}$};
					\draw [shift={(64.,10.)},line width=0.4pt]  plot[domain=5.257759132311878:5.982897107046088,variable=\t]({1.*4.037431066773457*cos(\t r)+0.*4.037431066773457*sin(\t r)},{0.*4.037431066773457*cos(\t r)+1.*4.037431066773457*sin(\t r)});
					\draw [shift={(64.,10.)},color=black]  plot[domain=0.21715133778930423:5.982897107046091,variable=\t]({1.*4.037431066773395*cos(\t r)+0.*4.037431066773395*sin(\t r)},{0.*4.037431066773395*cos(\t r)+1.*4.037431066773395*sin(\t r)});
					\begin{scriptsize}
					\draw [fill=black] (64.,10.) circle (1.5pt);
					\draw [fill=black] (64.,10.) circle (1.5pt);
					\draw [fill=black] (60.39320973638185,11.814363142597488) circle (1.5pt);
					\draw [fill=black] (60.04207935839705,9.202683366710461) circle (1.5pt);
					\draw [fill=black] (61.10367540176967,7.187144966296193) circle (1.5pt);
					\draw [fill=black] (68.,10.) circle (0.3pt);
					\draw [fill=black] (67.98,9.76) circle (0.3pt);
					\draw [fill=black] (63,6.1) circle (1.5pt);
					\draw [fill=black] (67.96341651367774,9.465941589917382) circle (0.3pt);
					\end{scriptsize}
					\end{tikzpicture}
					\hspace*{-1.4cm}
					\begin{tikzpicture}[line cap=round,line join=round,>=triangle 45,x=0.4cm,y=0.4cm]
					\clip(54.5,3.2) rectangle (68.5,14);
					\draw [color=black] (64.,10.)-- (60.04207935839705,9.202683366710461);
					\draw [color=black] (64.,10.)-- (61.10367540176967,7.187144966296193);
					\draw [color=black] (64.,10.)-- (63,6.1);
					\draw (64.6,6) node[anchor=north west] {\begin{scriptsize}4\end{scriptsize}};
					\draw (62.6,6) node[anchor=north west] {\begin{scriptsize}3\end{scriptsize}};
					\draw (59.2,9.98) node[anchor=north west] {\begin{scriptsize}1\end{scriptsize}};
					\draw (60.7,7.2) node[anchor=north west] {\begin{scriptsize}2\end{scriptsize}};
					\draw (63.2,11.2) node[anchor=north west] {\begin{scriptsize}t+4\end{scriptsize}};
					\draw (58.6,12.4) node[anchor=north west] {\begin{scriptsize}t+3\end{scriptsize}};
					\draw (62.5,4.5) node[anchor=north west] {$G_{(a)_3}$};
					\draw [shift={(64.,10.)},line width=0.4pt]  plot[domain=5.257759132311878:5.982897107046088,variable=\t]({1.*4.037431066773457*cos(\t r)+0.*4.037431066773457*sin(\t r)},{0.*4.037431066773457*cos(\t r)+1.*4.037431066773457*sin(\t r)});
					\draw [shift={(64.,10.)},color=black]  plot[domain=0.21715133778930423:5.982897107046091,variable=\t]({1.*4.037431066773395*cos(\t r)+0.*4.037431066773395*sin(\t r)},{0.*4.037431066773395*cos(\t r)+1.*4.037431066773395*sin(\t r)});
					\begin{scriptsize}
					\draw [fill=black] (65,6.1) circle (1.5pt);
					\draw [fill=black] (63,6.1) circle (1.5pt);
					\draw [fill=black] (64.,10.) circle (1.5pt);
					\draw [fill=black] (60.39320973638185,11.814363142597488) circle (1.5pt);
					\draw [fill=black] (60.04207935839705,9.202683366710461) circle (1.5pt);
					\draw [fill=black] (61.10367540176967,7.187144966296193) circle (1.5pt);
					\draw [fill=black] (68.,10.) circle (0.3pt);
					\draw [fill=black] (67.98,9.76) circle (0.3pt);
					\draw [fill=black] (67.96341651367774,9.465941589917382) circle (0.3pt);
					\end{scriptsize}
					\end{tikzpicture}
				\end{center}
\vspace*{-.2cm}				\caption{The  graphs $G_{(a)_i}$}\label{type a}
			\end{figure}
			
			\indent (b) Let $\bar{G}$ be  the graph  $G_{(b)}$  in  Figure~\ref{type b}, where $t\geq 1$ and  $\{i,t+4\}\in E({\bar{G}})$ for all $i\in[t+3]$. Then $\bar{G}$ has one minimal cycle $C=1-2-\cdots-(t+3)-1$ as in (a). Setting $W=[t+3]\cup\{t+5\}$, we have $T(C)\in \ker \partial_1^W\setminus \im \partial_2^W$. It follows that $\beta_{t+1,t+4}(I)\neq 0$. Therefore $\projdim(I)\geq t+1$. For any $W\subseteq [t+5]$ with $|W|=t+4$, $\bar{G}_W$ is connected. So $\beta_{t+2,t+4}(I)=0$. Suppose $W=[t+5]$. Although  $T(C)\in \ker \partial_1^W$ one also has $T(C)\in  \im \partial_2^W$, because $C$ is the base of a cone with apex ${t+4}$. Hence according to (O-2), 			 
			 $\beta_{t+2,t+5}(I)=0$. It follows that $\projdim(I)=t+1$.
			\begin{figure}[ht!]
				\begin{center}
					\hspace*{-1.8cm}
					\begin{tikzpicture}[line cap=round,line join=round,>=triangle 45,x=0.4cm,y=0.4cm]
					\clip(54.5,5.8) rectangle (68.5,14.1);
					\draw [color=black] (64.,10.)-- (60.04207935839705,9.202683366710461);
					\draw [color=black] (64.,10.)-- (61.10367540176967,7.187144966296193);
					\draw (59.2,9.98) node[anchor=north west] {\begin{scriptsize}1\end{scriptsize}};
					\draw (60.7,7.2) node[anchor=north west] {\begin{scriptsize}2\end{scriptsize}};
					\draw (63.2,11.2) node[anchor=north west] {\begin{scriptsize}t+4\end{scriptsize}};
					\draw (58.6,12.4) node[anchor=north west] {\begin{scriptsize}t+3\end{scriptsize}};
					\draw [shift={(64.,10.)},line width=0.4pt]  plot[domain=5.257759132311878:5.982897107046088,variable=\t]({1.*4.037431066773457*cos(\t r)+0.*4.037431066773457*sin(\t r)},{0.*4.037431066773457*cos(\t r)+1.*4.037431066773457*sin(\t r)});
					\draw [color=black] (64.,10.)-- (60.39320973638185,11.814363142597488);
					\draw [color=black] (56.47518012335943,6.228129663627268)-- (60.04207935839705,9.202683366710461);
					\draw [color=black] (56.47518012335943,6.228129663627268)-- (61.10367540176967,7.187144966296193);
					\draw (54.7,6.7) node[anchor=north west] {\begin{scriptsize}t+5\end{scriptsize}};;
					\draw [shift={(64.,10.)},color=black]  plot[domain=0.21715133778930423:5.982897107046091,variable=\t]({1.*4.037431066773395*cos(\t r)+0.*4.037431066773395*sin(\t r)},{0.*4.037431066773395*cos(\t r)+1.*4.037431066773395*sin(\t r)});
					\draw [color=black] (64.,10.)-- (65.63698868515817,9.130168501265073);
					\draw [color=black] (64.,10.)-- (65.66097247720477,10.689114984293646);
					\begin{scriptsize}
					\draw [fill=black] (64.,10.) circle (1.5pt);
					\draw [fill=black] (60.39320973638185,11.814363142597488) circle (1.5pt);
					\draw [fill=black] (60.04207935839705,9.202683366710461) circle (1.5pt);
					\draw [fill=black] (61.10367540176967,7.187144966296193) circle (1.5pt);
					\draw [fill=black] (56.47518012335943,6.228129663627268) circle (1.5pt);
					\draw [fill=black] (68.,10.) circle (0.3pt);
					\draw [fill=black] (67.98,9.76) circle (0.3pt);
					\draw [fill=black] (67.96341651367774,9.465941589917382) circle (0.3pt);
					\draw [fill=black] (65.46910214083202,10.137487767221996) circle (0.3pt);
					\draw [fill=black] (65.47,9.93) circle (0.3pt);
					\draw [fill=black] (65.44511834878543,9.729763302429909) circle (0.3pt);
					\end{scriptsize}
					\end{tikzpicture}
				\end{center}
\vspace*{-.2cm}				\caption{The  graph $G_{(b)}$}\label{type b}
			\end{figure}
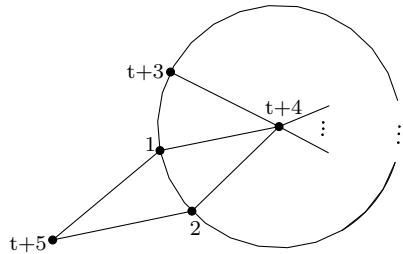

					\indent (c) Let $\bar{G}$ be the graph $G_{(c)}$ shown in Figure~\ref{type c}. This graph consists of $3$ minimal cycles of length $4$. So $\Index(I)=1$. Moreover, $\beta_{2,5}(I)\neq 0$ because
			$T(C)\in \ker \partial_1\setminus \im\partial_2$, for all minimal cycles $C$ in $\bar{G}$, and $\beta_{3,5}(I)=0$ because $\bar{G}$ is connected. Hence  $\projdim(I)=\!2$.
			\begin{figure}[ht!]
				\begin{center}
					\begin{tikzpicture}[line cap=round,line join=round,>=triangle 45,x=0.6cm,y=0.6cm]
					\clip(7.,4) rectangle (12.5,8.1);
					\draw (8.,4.)-- (12.,4.);
					\draw (12.,4.)-- (12.,8.);
					\draw (12.,8.)-- (8.,8.);
					\draw (8.,8.)-- (8.,4.);
					\draw (12.,8.)-- (8.,4.);
					\draw (7.3,4.553017458995247) node[anchor=north west] {\begin{scriptsize}1\end{scriptsize}};
					\draw (12,4.553017458995247) node[anchor=north west] {\begin{scriptsize}2\end{scriptsize}};
					\draw (9.8,6.84034134932995) node[anchor=north west] {\begin{scriptsize}5\end{scriptsize}};
					\draw (7.3,8.3) node[anchor=north west] {\begin{scriptsize}4\end{scriptsize}};
					\draw (12,8.3) node[anchor=north west] {\begin{scriptsize}3\end{scriptsize}};
					\begin{scriptsize}
					\draw [fill=black] (8.,4.) circle (1.5pt);
					\draw [fill=black] (12.,4.) circle (1.5pt);
					\draw [fill=black] (12.,8.) circle (1.5pt);
					\draw [fill=black] (8.,8.) circle (1.5pt);
					\draw [fill=black] (10.,6.) circle (1.5pt);
					\end{scriptsize}
					\end{tikzpicture}
				\end{center}
				\caption{The  graph $G_{(c)}$}\label{type c}
			\end{figure}
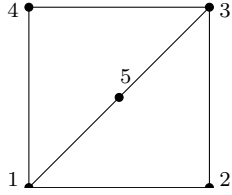
			
			\medskip
			(d) Let $\bar{G}$ be either of the  graphs $G_{(d)_1}, G_{(d)_2}$  in Figure~\ref{type d}. 
			Both graphs have three minimal cycles of length~$4$. Since  $G_{(d)_1}$ is a dipyramid,  
			  by (O-3) one has $\widetilde{H}_2(\Delta(\overline{G_{(d)_1}});\mathbb{K})\neq 0$ which implies that $\beta_{2,6}(I(\overline{G_{(d)_1}}))\neq 0$. Although, $G_{(d)_2}$ is not a dipyramid, it contains the minimal cycle $C=1-2-3-4-1$ which gives a nonzero homology class of $\widetilde{H}_1(\Delta(\overline{G_{(d)_2}})_W;\mathbb{K})\neq 0$, where $W=V(C)\cup\{5\}$. Hence $\beta_{2,5}(I(\overline{G_{(d)_2}}))\neq 0$. 
						Therefore  $\projdim(I)\geq 2$ in both cases. 
			To prove that  $\projdim(I)=2$  it is enough to show that $\beta_{3,5}(I)=\beta_{3,6}(I)=0$.
			
			Considering any subset $W$ of $[6]$ with $|W|=5$, $\bar{G}_W$ and so $\Delta_W$ is connected in both cases. It follows that $\widetilde{H}_0(\Delta_W;\mathbb{K})=0$, and hence $\beta_{3,5}(I)=0$. Now $\widetilde{H}_1(\Delta;\mathbb{K})=0$ because except for the cycle $C=1-2-3-4-1$ in $G_{(d)_2}$, all other minimal cycles in $G_{(d)_1}, G_{(d)_2}$ are bases of some cones and for the cycle $C$, we have 
			$$T(C)=\partial_2^W([1,2,5]+[2,3,5]+[3,4,6]-[1,4,6]-[3,5,6]+ [1,5,6]).$$ 
			Consequently, $\beta_{3,6}(I)=0$ and hence $\projdim(I)=2$.
		
		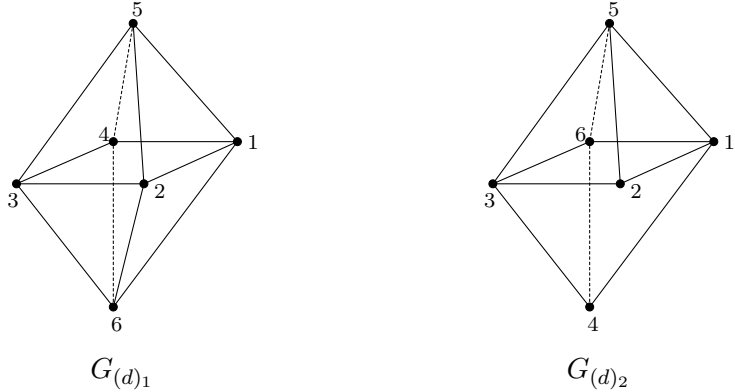
\begin{figure}[ht!]
			\begin{center}
				\hspace*{-.5cm}
				\begin{tikzpicture}[line cap=round,line join=round,>=triangle 45,x=0.55cm,y=0.55cm]
				\clip(5.,0) rectangle (11.5,9.4);
				\draw  (8.,6.)-- (11.,6.);
				\draw (11.,6.)-- (8.74,4.98);
				\draw  (8.74,4.98)-- (5.66,4.98);
				\draw  (5.66,4.98)-- (8.,6.);
				\draw  (8.48,8.86)-- (11.,6.);
				\draw   (8.48,8.86)-- (8.74,4.98);
				\draw   (8.48,8.86)-- (5.66,4.98);
				\draw [dash pattern=on 1pt off 1pt] (8.48,8.86)-- (8.,6.);
				\draw   (8.,2.)-- (8.74,4.98);
				\draw   (8.,2.)-- (5.66,4.98);
				\draw [dash pattern=on 1pt off 1pt] (8.,2.)-- (8.,6.);
				\draw   (8.,2.)-- (11.,6.);
				\draw (5.2,4.99) node[anchor=north west] {\begin{scriptsize}3\end{scriptsize}};
				\draw (8.74,5.22) node[anchor=north west] {\begin{scriptsize}2\end{scriptsize}};
				\draw (11.,6.4) node[anchor=north west] {\begin{scriptsize}1\end{scriptsize}};
				\draw (7.4,6.6) node[anchor=north west] {\begin{scriptsize}4\end{scriptsize}};
				\draw (8.2,9.62) node[anchor=north west] {\begin{scriptsize}5\end{scriptsize}};
				\draw (7.7,2.) node[anchor=north west] {\begin{scriptsize}6\end{scriptsize}};
				\draw (7.2,1) node[anchor=north west] {$G_{(d)_1}$};
				
				\begin{scriptsize}
				\draw [fill=black] (8.,6.) circle (1.5pt);
				\draw [fill=black] (11.,6.) circle (1.5pt);
				\draw [fill=black] (8.74,4.98) circle (1.5pt);
				\draw [fill=black] (5.66,4.98) circle (1.5pt);
				\draw [fill=black] (8.48,8.86) circle (1.5pt);
				\draw [fill=black] (8.,2.) circle (1.5pt);
				\end{scriptsize}
				\end{tikzpicture}
				\hspace{2.5cm}
				\begin{tikzpicture}[line cap=round,line join=round,>=triangle 45,x=0.55cm,y=0.55cm]
				\clip(5.,0) rectangle (11.5,9.5);
				\draw   (8.,6.)-- (11.,6.);
				\draw   (11.,6.)-- (8.74,4.98);
				\draw   (8.74,4.98)-- (5.66,4.98);
				\draw   (5.66,4.98)-- (8.,6.);
				\draw   (8.48,8.86)-- (11.,6.);
				\draw   (8.48,8.86)-- (8.74,4.98);
				\draw   (8.48,8.86)-- (5.66,4.98);
				\draw [dash pattern=on 1pt off 1pt] (8.48,8.86)-- (8.,6.);
				\draw   (8.,2.)-- (5.66,4.98);
				\draw [dash pattern=on 1pt off 1pt] (8.,2.)-- (8.,6.);
				\draw   (8.,2.)-- (11.,6.);
				\draw (5.2,4.99) node[anchor=north west] {\begin{scriptsize}3\end{scriptsize}};
				\draw (8.74,5.22) node[anchor=north west] {\begin{scriptsize}2\end{scriptsize}};
				\draw (11.,6.4) node[anchor=north west] {\begin{scriptsize}1\end{scriptsize}};
				\draw (7.4,6.6) node[anchor=north west] {\begin{scriptsize}6\end{scriptsize}};
				\draw (8.2,9.62) node[anchor=north west] {\begin{scriptsize}5\end{scriptsize}};
				\draw (7.7,2.) node[anchor=north west] {\begin{scriptsize}4\end{scriptsize}};
				\draw (7.2,1) node[anchor=north west] {$G_{(d)_2}$};

				\begin{scriptsize}
				\draw [fill=black] (8.,6.) circle (1.5pt);
				\draw [fill=black] (11.,6.) circle (1.5pt);
				\draw [fill=black] (8.74,4.98) circle (1.5pt);
				\draw [fill=black] (5.66,4.98) circle (1.5pt);
				\draw [fill=black] (8.48,8.86) circle (1.5pt);
				\draw [fill=black] (8.,2.) circle (1.5pt);
				\end{scriptsize}
				\end{tikzpicture}
			\end{center}
		\vspace*{-.3cm}	\caption{The graphs $G_{(d)_i}$}\label{type d}
		\end{figure}
	\end{ex}

		 Next lemma gives more intuition about the length of minimal cycles in $\bar{G}$, when $I(G)$ has almost maximal finite index. For an integer $k$,  we show by  $\overline{k}$ the remainder of $k$ modulo $t+3$, i.e. $\overline{k}\equiv k\pmod{t+3}$ with $0 \leq \overline{k} < t + 3$, where $t\geq 1$ is an integer.

		\begin{lem}\label{sedaghashang}
			Let  $G$ be a simple graph on $[n]$. Assume  $I:=I(G)$ has almost maximal finite index. Then any  minimal cycle in $\bar{G}$ is of length $\Index(I)+3$.
		\end{lem}

		\begin{proof}
			Let $\Index(I)=t$. Then $\projdim(I)=t+1$ which means that $\beta_{i,j}(I)=0$ for all $i>t+1$ and all $j$.
			Using Theorem~\ref{index of graphs}, there exists a minimal cycle $C$ in $\bar{G}$ of length $t+3$ which has the  smallest length among all the minimal cycles in $\bar{G}$. Let $C=u_1-u_2-\cdots-u_{t+3}-u_1$. Suppose $C'\neq C$ is a minimal cycle in $\bar{G}$ with $C'=v_1-v_2-\cdots-v_l-v_1$. Setting $W=V(C')$  and $T(C')$ as defined in (\ref{cycle kernel}) one has 
					$T(C')\in \ker \partial_1^W$, while $\im\partial_2^W=0$. Hence $\widetilde{H}_{1}(\Delta_W;\mathbb{K})\neq 0$.  Hochster's formula implies that $\beta_{l-3,l}(I)\neq 0$. Since $\beta_{i,j}(I)=0$ for all $i>t+1$, we have  $l\leq t+4$. We claim that $l<t+4$. Since $t+3$ is the smallest length of a minimal cycle in $\bar{G}$, it follows that $l=t+3$, as desired. 
			
			\medskip
			{\em Proof of the claim:} Suppose $l=t+4$ and let $u\in [n]\setminus V(C')$.  Note that such $u$ exists since otherwise $V(C)\subset [n]=V(C')$ which implies that $C'$ is not minimal. Let $W=V(C')\cup\{u\}$. 			
			Since $\beta_{t+2,t+5}(I)= 0$, it follows that  $\widetilde{H}_{1}(\Delta_W;\mathbb{K})= 0$. Therefore, $T(C')\in \ker \partial_1^W$ implies that  $T(C')\in \im \partial_2^W$. Hence,  $\{u, v_i\}\in E(\bar{G})$ for all $1\leq i\leq t+4$.  
				  On the other hand, since $n>t+4$  there exist $v, v'\in [n]\setminus V(C)$ with $v\neq v'$. Setting $W=V(C)\cup\{v,v'\}$, by \textbf{(O-4)}(ii), either of the following cases happens:
				\begin{itemize}
				\item[(i)]  either $\{v,u_i\}\in E(\bar{G})$  for all $u_i\in V(C)$ or $\{v',u_i\}\in E(\bar{G})$ for all $u_i\in V(C)$; 
				\item[(ii)] else, $t=1$ and $\Delta_W$ is isomorphic to the graph $G_{(d)_2}$ in Figure~\ref{type d}. In particular, $\{v,v'\}\in E(\bar{G})$.
				\end{itemize}

		We show that $V(C)\cap V(C')= \emptyset$. Suppose on contrary that  $V(C)\cap V(C')\neq \emptyset$, say  $u_1\in V(C')$. Then $\{u_j, u_1\}\in E(\bar{G})$ for all $u_j\in V(C)\setminus V(C')$. Since $C$ is minimal we conclude that  $V(C)\setminus V(C')\subseteq \{u_2,u_{t+3}\}$. Therefore $\{u_1,u_3,\ldots,u_{t+2}\}\subset V(C')$. 
	 Note that $V(C)\setminus V(C')\neq \emptyset$ because otherwise $V(C)\subset V(C')$ which does not hold.  
		
		If $|V(C)\setminus V(C')|=1$, without loss of generality we may suppose $V(C)\setminus V(C')=\{u_{t+3}\}$. Then since $|V(C')|-|V(C)|=1$, we have $|V(C')\setminus V(C)|=2$.  Let $v_{j_1},v_{j_2}\in V(C')\setminus V(C)$. Then $V(C')=\{v_{j_1},v_{j_2}\}\cup\{u_1,\ldots,u_{t+2}\}$ with $\{v_{j_1},v_{j_2}\}\cap\{u_1,\ldots,u_{t+2}\}=\emptyset$.  Suppose (i) happens for $v_{j_1},v_{j_2}$. We may assume that $\{v_{j_1}, u_i\}\in E(\bar{G})$ for all $1\leq i\leq t+3$. Since $t\geq 1$, $|\{u_1,\ldots,u_{t+2}\}|\geq 3$ which implies that $v_{j_1}$ is adjacent to at least $3$ vertices of $C'$ in $\bar{G}$ which contradicts the minimality of $C'$.  So (i) cannot happen when $|V(C)\setminus V(C')|=1$. Therefore by (ii), $t=1$ and the induced subgraph of $\bar{G}$ on $V(C)\cup\{v_{j_1},v_{j_2}\}$  is isomorphic to $G_{(d)_2}$. Since $V(C')\subset V(C)\cup\{v_{j_1},v_{j_2}\}$, the cycle $C'$ which is of length $5$ is an induced subgraph of  $G_{(d)_2}$. This is  a contradiction because all cycles in $G_{(d)_2}$ are of length $4$. Therefore $|V(C)\setminus V(C')|=2$.

 It follows from $V(C)\setminus V(C')=\{u_2,u_{t+3}\}$ that  $V(C')=\{v_{j_1},v_{j_2}, v_{j_3}\}\cup\{u_1,u_3,\ldots,u_{t+2}\}$ with $\{v_{j_1},v_{j_2}, v_{j_3}\}\cap\{u_1,u_3,\ldots,u_{t+2}\}=\emptyset$. If at least two of the vertices $v_{j_1},v_{j_2},v_{j_3}$, say $v_{j_1},v_{j_2}$, are adjacent to all vertices of $C$ in $\bar{G}$, then $v_{j_1}-u_1-v_{j_2}-u_3-v_{j_1}$ is a $4$-cycle in $C'$ which contradicts the minimality of $C'$, because $|C'|\geq 5$.  Hence at most one vertex from $v_{j_1},v_{j_2},v_{j_3}$ is adjacent to all vertices of $C$ in $\bar{G}$.  If none of them is adjacent to all $u_i$ in $\bar{G}$, by (ii) we have $\{v_{j_1},v_{j_2}\},\{v_{j_2},v_{j_3}\},\{v_{j_1},v_{j_3}\}\in E(\bar{G})$ and hence   $C'$ contains a triangle which is a contradiction. Therefore, exactly one vertex among $v_{j_1},v_{j_2},v_{j_3}$, say $v_{j_1}$,  is adjacent to all vertices $u_i$ in $\bar{G}$.  Now 
 $\{u_1,u_3,\ldots,u_{t+2}\}\subset V(C')$  and minimality of $C'$ imply that   $t=1$ and that $v_{j_1}$ is not adjacent to $v_{j_2},v_{j_3}$. 
 
  Setting $W=\{v_{j_2},v_{j_3}\}\cup V(C)$, since (i) does not happen for this $W$,  one concludes that $\Delta_W$ is isomorphic to $G_{(d)_2}$. Therefore $\{v_{j_2},v_{j_3}\}\in E(\bar{G})$  and, in  $\bar{G}$ the vertex   $v_{j_2}$ is adjacent to  three successive vertices $u_{\overline{i-1}}, u_{i}, u_{\overline{i+1}}$ of $C$,  and the vertex $v_{j_3}$ is adjacent to $u_{\overline{i+1}}, u_{\overline{i+2}}, u_{\overline{i-1}}$, where $1\leq i\leq 4$. Since $\{v_{j_2},v_{j_3}\}\in E(C')$,   and since $V(C')=\{v_{j_1}, v_{j_2},v_{j_3},u_1,u_3\}$, 
  it follows that either $i=1$ or $i=3$, otherwise $C'$ is not minimal. Without loss of generality suppose $i=1$. Thus $v_{j_2}$ is adjacent to $u_1,u_2,u_4$ but  not to $u_3$, and $v_{j_3}$ is adjacent to $u_2,u_3,u_4$ but not to $u_1$. 
  Setting $W=V(C)\cup\{v_{j_1},v_{j_2},v_{j_3}\}$ one has 
 \begin{align*}
T'=(&\sum_{1\leq i\leq 3}[v_{j_1},u_i,u_{i+1}])-[v_{j_1},u_1,u_4]- [v_{j_2},u_1,u_{2}]+[v_{j_2},u_1,u_{4}]-[v_{j_3},u_2,u_{3}]\\
&-[v_{j_3},u_3,u_4]+[v_{j_2},v_{j_3},u_2]-[v_{j_2},v_{j_3},u_4]\in\ker\partial_2^W.
\end{align*}
 Since $\beta_{3,7}(I)=0$, we have  $T'\in  \im \partial_3^W$ which requires that $\Delta_W$ contains faces of dimension $3$ which is not the case here, a contradiction. 
  Consequently, $V(C)\cap V(C')= \emptyset$, as desired.

  	Setting $W=\{u_j\}\cup V(C')$ for some $1\leq j\leq t+3$, since $T(C')\in \ker\partial_1^W$ and $\beta_{t+2,t+5}(I)=0$ we conclude that $u_j$ is adjacent to all vertices of $C'$ in $\bar{G}$. In particular, 
			$\{u_1,v_i\}, \{u_3,v_i\}\in E(\bar{G})$ for all $1\leq i\leq t+4$. Let $W=V(C')\cup\{u_1,u_3\}$. Then $\Delta_W$ consists of an induced dipyramid $D$. 
			Thus  $T(D)\in \ker \partial_2^W$, while $\im \partial_3^W=0$, where $T(D)$ is defined in (\ref{dipyramid kernel}). 
			It follows that $\widetilde{H}_{2}(\Delta_W;\mathbb{K})\neq 0$ and so $\beta_{t+2,t+6}(I)\neq 0$, a contradiction. Therefore $l<t+4$ and the claim follows.
		\end{proof}

		In the next corollary we highlight some information obtained from Observation~\ref{connected induced graphs}  about the vertices not belonging to a minimal cycle.
		\begin{cor}\label{ostad}
			Let  $G$ be a simple graph on $[n]$. Assume   $I:=I(G)$ has almost maximal finite index. Let   $C$ be a minimal cycle in $\bar{G}$. Then 
			
			\begin{itemize}
				\item[(a)] all vertices in $[n]\setminus V(C)$ are  adjacent to some vertex in $V(C)$ in the graph $\bar{G}$.
				\item[(b)] For any pair of vertices $v,v'\in [n]\setminus V(C)$  whenever  $|N_{\bar{G}}(v)\cap V(C)|\leq 2$, then $V(C)\subseteq N_{\bar{G}}(v')$.
				\item[(c)] If $\Index(I)=1$, then there are at most two vertices in $[n]\setminus V(C)$ which are not adjacent to all vertices of ${C}$ in  $\bar{G}$.
				\item[(d)] If $\Index(I)>1$, then  there is at most one vertex in $[n]\setminus V(C)$ which is not adjacent to all vertices of ${C}$ in  $\bar{G}$.
			\end{itemize} 
		\end{cor}
		\begin{proof}
			Let $\Index(I)=t$. By assumption $\projdim(I)=t+1$. By Lemma~\ref{sedaghashang} all minimal cycles of $\bar{G}$ are of length $t+3$. Let $C=u_1-u_2-\cdots-u_{t+3}-u_1$ be a minimal cycle of $\bar{G}$.  
			
			(a) Let $u_{t+4}\in [n]\setminus V(C)$, and set $W=V(C)\cup \{u_{t+4}\}$. Since $\beta_{t+2,t+4}(I)=0$ we conclude that $\bar{G}_W$ is connected using (O-1). It follows that $u_{t+4}$ is adjacent to some vertex of $C$ in the graph $\bar{G}$. 
			
			\medspace
			(b) 	If $|[n]\setminus V(C)|\leq 1$, then there is nothing to prove. Suppose  $u_{t+4},u_{t+5}\in [n]\setminus V(C)$. Set $W= V(C)\cup\{u_{t+4},u_{t+5}\}$.  Since $\beta_{t+2,t+5}(I)=0$, (O-4)(ii) implies that for each edge $e$ of $C$ we either have $e\cup\{u_{t+4}\}\in \Delta_W$ or  $e\cup\{u_{t+5}\}\in \Delta_W$. This in particular shows that if $u_{t+4}$ is adjacent to at most $2$  vertices of $C$ in $\bar{G}$, then $u_{t+5}$ is adjacent to all of them in $\bar{G}$. 
		
		\medspace	
			(c) Suppose $u_{5},u_{6}\in [n]\setminus V(C)$ are not adjacent to all vertices of $C$ in $\bar{G}$. The argument in (O-4)(ii) shows that we may assume that $\{u_1,u_2,u_3\}\subseteq N_{\bar{G}}(u_5)$ but $u_4\notin N_{\bar{G}}(u_5)$ and $\{u_1,u_3,u_4\}\subseteq N_{\bar{G}}(u_6)$ but $u_2\notin N_{\bar{G}}(u_6)$. Now suppose $u_7\in [n]\setminus V(C)$ is not adjacent to all vertices of $C$ in $\bar{G}$. By replacing $u_6$ with $u_7$ in (O-4)(ii) one sees that $u_7$ is not adjacent to $u_2$ in $\bar{G}$, and replacing $u_5$ with $u_7$ in the same argument shows that $u_7$ is adjacent to $u_2$ in $\bar{G}$, a contradiction. 
				
			\medspace
			(d) Suppose $u_{t+4},u_{t+5}$ are two vertices in $[n]\setminus V(C)$ which are not adjacent to all vertices of $C$ in $\bar{G}$. The argument in (O-4)(ii) shows that $t=1$, a contradiction. 
		\end{proof}	
The crucial point in the classification of the edge ideals with almost maximal finite index is to determine the number of the vertices of the graph with respect to the index of the ideal.  In the following, we compute this number.

\begin{prop}\label{number of vertices}
Let $G$ be a simple graph on $[n]$ with no isolated vertex such that $I=I(G)$ has almost maximal finite index $t$.   Then $G$ has either $n=t+4$ or $n=t+5$ vertices.
\end{prop}

\begin{proof}
Since $\Index(I)=t$ there is a minimal cycle $C=u_1-u_2-\cdots-u_{t+3}-u_1$ in $\bar{G}$. Moreover, $\bar{G}\neq C$, because otherwise $\projdim(I)=\Index(I)$ by \cite[Theorem~4.1]{BHZ}. Since $C$ is a minimal cycle, $\bar{G}\neq C$ means that  there exists $v\in [n]\setminus V(C)$. Therefore $n\geq t+4$. 

Suppose on contrary that $n> t+5$. So $n-|V(C)|>2$. 

Suppose first  $t>1$. It follows from Corollary~\ref{ostad}(d) that there exist  $u_{t+4},u_{t+5}\in [n]\setminus V(C)$ such that  $u_{t+4},u_{t+5}$ are adjacent to all vertices of $C$ in $\bar{G}$. Therefore  $C':=u_{t+4}-u_1-u_{t+5}-u_3-u_{t+4}$ is a $4$-cycle. Since $t>1$, $C'$ is not minimal and hence $\{u_{t+4},u_{t+5}\}\in E(\bar{G})$. 

\medspace
 Since $u_{t+4}, u_{t+5}$ are not isolated in $G$, there exist $v_1, v_2\in  [n]\setminus (V(C)\cup \{ u_{t+4}, u_{t+5}\})$ such that $\{v_1, u_{t+4}\}, \{v_2,u_{t+5}\}\notin E(\bar{G})$. By Corollary~\ref{ostad}(a), $v_1, v_2$ are adjacent to some vertices  of $C$ in $\bar{G}$.  If $v_1$ is adjacent to at least two vertices in $\bar{G}$, say $u_a, u_b\in V(C)$  such that   $b\neq \overline{a+1}$ and $a\neq \overline{b+1}$, then we will have a minimal $4$-cycle $v_1-u_a-u_{t+4}-u_b-v_1$ which contradicts $t>1$. Thus $v_1$ is adjacent to either one vertex $u_a$  or  two vertices $u_a, u_{\overline{a+1}}$ of $C$ in $\bar{G}$. In particular, $v_1$ is not adjacent to all vertices of $C$ in $\bar{G}$. Same holds for $v_2$. Corollary~\ref{ostad}(d) implies that $v_1=v_2$.  
  If $v_1$ is adjacent to only one vertex $u_a$ of $C$ in $\bar{G}$, setting $W=\{u_{t+4},v_1\}\cup V(C)\setminus \{u_a\}$, $\Delta_W$ is not connected and so $\beta_{t+2,t+4}(I)\neq 0$, a contradiction.  Therefore, $v_1$ is adjacent to $u_a, u_{\overline{a+1}}$ in $\bar{G}$. Setting $W=\{u_{t+4},u_{t+5},v_1\}\cup V(C)\setminus \{u_a, u_{\overline{a+1}}\}$, $\Delta_W$ is not connected and so $\beta_{t+2,t+4}(I)\neq 0$, a contradiction. Consequently,  $n\leq t+5$ when $t>1$.

\medspace
Now suppose $t=1$. Since $n-|V(C)|>2$ we have $n\geq 7$. By Corollary~\ref{ostad}(c), at least one vertex, say $v_1$ in $[n]\setminus V(C)$ is adjacent to all vertices of $C$ in $\bar{G}$.  Since $v_1$ is not isolated in $G$, there exists $v_2\in [n]\setminus(V(C)\cup \{v_1\})$ such that $\{v_1,v_2\}\notin E(\bar{G})$. We claim that $v_2$ is not adjacent to some vertex of $C$ in $\bar{G}$. 

\medspace
{\em Proof of the claim:} Suppose on contrary that $v_2$ is adjacent to all vertices of $C$ in $\bar{G}$. Then we get an induced  dipyramid $D$ on the vertex set $V(C)\cup\{v_1,v_2\}$. Now set $W=V(C)\cup \{v_1,v_2,v_3\}$ for some $v_3\in [n]\setminus(V(C)\cup \{v_1,v_2\})$. Since $\beta_{3,7}(I)=0$ we have $T(D)\in \im\partial_{3}^W$, with $T(D)$ similar to the one in (\ref{dipyramid kernel}), which implies that each $2$-face of $D$ is contained in a $3$-face of $\Delta_W$ and hence $v_3$ is adjacent to all vertices of $V(C)\cup \{v_1,v_2\}$ in $\bar{G}$.  As $v_3$ is not isolated in $G$, there exists $v_4\in [n]\setminus(V(C)\cup \{v_1,v_2,v_3\})$ such that $\{v_3,v_4\}\notin E(\bar{G})$. Replacing $v_3$ with $v_4$ in the above argument  we conclude that $v_4$ is also adjacent to all vertices of $V(C)\cup \{v_1,v_2\}$ in $\bar{G}$. 
Now
set $W=\{v_1,v_2,v_3,v_4\}\cup V(C)$. Then  
\vspace*{-.55cm}
\begin{align*}
T=&\sum_{1\leq i\leq 3}\left([v_1,v_3,u_i,u_{i+1}]-[v_1,v_4,u_i,u_{i+1}]-[v_2,v_3,u_i,u_{i+1}]+[v_2,v_4,u_i,u_{i+1}]\right)\\
&-[v_1,v_3,u_1,u_4]+[v_1,v_4,u_1,u_4]+[v_2,v_3,u_1,u_4]-[v_2,v_4,u_1,u_4]\in\ker\partial_3^W
\end{align*}
while $T\notin\im\partial_{4}^W$, because $\Delta_W$ contains no $4$-face.  
This implies that $\beta_{3,8}(I)\neq 0$ which is a contradiction. So the claim follows.
\vspace{0cm}

Without loss of generality suppose  $\{v_2,u_4\}\notin E(\bar{G})$. Now consider $v'_3\in [n]\setminus (V(C)\cup\{v_1,v_2\})$. We  show that  $v'_3$ is   adjacent to all vertices of $C$ in $\bar{G}$. Otherwise, setting $W=\{v_2,v'_3\}\cup V(C)$ the same discussion as in (O-4)(ii) shows that $\bar{G}_W$ is isomorphic to the graph $G_{(d)_2}$ in Figure~\ref{type d}, where $\{v'_3,u_2\}\notin \bar{G}_W$. Hence  setting $W=\{v_1,v_2,v'_3\}\cup V(C)$, we have
\begin{align*}
T'=(&\sum_{1\leq i\leq 3}[v_1,u_i,u_{i+1}])-[v_1,u_1,u_4]- [v_2,u_1,u_{2}]-[v_2,u_2,u_{3}]-[v'_3,u_3,u_{4}]\\
&+[v'_3,u_1,u_4]-[v_2,v'_3,u_1]+[v_2,v'_3,u_3]\in\ker\partial_2^W,
\end{align*}
while $T'\notin \im\partial^W_3$ because $\{v_2,u_1,u_2\}$ is not contained in a $3$-face of $\Delta_W$, and we get a contradiction. Thus 
 $v'_3$ is adjacent to all vertices of $C$ in $\bar{G}$. It follows that  setting $W=\{v_1,v_2,v'_3\}\cup V(C)$, a dipyramid $D$ with the vertex set $V(C)\cup \{v_1,v'_3\}$ lies in $\Delta_W$ and so $T(D)\in \ker\partial_{2}^W$ which implies that $T(D)\in \im\partial_{3}^W$.  Thus each $2$-face of $D$ is contained in a $3$-face of $\Delta_W$. Since $v_2$ is not adjacent to $u_4$ in $\bar{G}$, we conclude that $\{v_1,v'_3\}\in E(\bar{G})$. 
Note that by $\beta_{3,5}(I)= 0$, setting $W=V(C)\cup \{v_2\}$, the vertex $v_2$ is adjacent to some vertex $u_i$ of $V(C)$ in $\bar{G}$. Now setting $W=\{v_1,v_2\}\cup V(C)\setminus\{u_i\}$, the same reason implies that $v_2$ is adjacent to some vertex $u_j$ in $V(C)\setminus\{u_i\}$ in the graph $\bar{G}$.  Finally, setting $W=\{v_1,v_2, v'_3\}\cup V(C)\setminus \{u_i,u_j\}$ indicates that in the graph $\bar{G}$ the vertex $v_2$ is adjacent to either three vertices $u_i,u_j,u_k$ of $C$ or it is adjacent to the two vertices $u_i,u_j$ of $C$ and to $v'_3$.  We show that  in the first case $v_2$ is also adjacent to $v'_3$ in $\bar{G}$. Suppose the first case happens. Since $\{v_2,u_4\}\notin E(\bar{G})$, setting $W=\{v_1,v_2,v'_3,u_1,u_3,u_4\}$ we have a minimal cycle $C':=v_2-u_3-u_4-u_1-v_2$ in $\bar{G}_W$ with $T(C')\in \ker\partial_1^W$. Since $\beta_{3,6}(I)=0$, any edge of $C'$ must be contained in a $2$-face of $\Delta_W$ and since $\{v_1,v_2\}\notin E(\bar{G})$ it follows that $v_2$ is adjacent to $v'_3$ in $\bar{G}$. 

Now since $v'_3$ is not isolated in $G$, there exists $v'_4\in [n]\setminus (V(C)\cup \{v_1,v_2,v'_3\})$ with $\{v'_3,v'_4\}\notin E(\bar{G})$. Replacing $v'_3$ with $v'_4$ in the above discussion, we see that $v'_4$ is adjacent to all vertices of $C$ in $\bar{G}$. Setting $W=V(C)\cup \{v_2,v'_3,v'_4\}$, we have an induced dipyramid $D$ on the vertex set $V(C)\cup\{v'_3,v'_4\}$ with $T(D)\in \ker\partial^W_{2}$ and since $\{v_2,u_4\}\notin \bar{G}_W$ we have  $T(D)\notin \im\partial^W_{3}$ that is a contradiction with $\beta_{3,7}(I)=0$. 
  Thus $n\leq t+5$ when $t=1$.
\end{proof}

Now we are ready to state the main result of this section which determines the graphs whose edge ideals have almost maximal finite index.
\begin{thm}\label{check-out}
Let $G$ be a simple graph on $[n]$ with no isolated vertex and let $I=I(G)\subset\!S$. Then $I$ has almost maximal finite index if and only if $\bar{G}$ is isomorphic to one of the graphs given in Example~\ref{mesal}.
\end{thm}

\begin{proof}
	The ``if" implication follows from Example~\ref{mesal}. We prove the converse. 
 Suppose $\Index(I)=t$. Then  there is a minimal cycle $C:=u_1-u_2-\cdots-u_{t+3}-u_1$ in $\bar{G}$. Moreover, by Proposition~\ref{number of vertices} there exists $u_{t+4}\in [n]\setminus V(C)$ which by Corollary~\ref{ostad}(a) is adjacent to some vertex $u_i$ of $V(C)$ in $\bar{G}$.  Without loss of generality we may assume that $i=1$. By Proposition~\ref{number of vertices} we have $n-(t+3)\leq 2$. We consider two cases: 

\medspace
{\em Case} (i): Suppose  $[n]\setminus V(C)=\{u_{t+4}\}$ and let  $1\leq l\leq t+3$ be the largest integer such that $\{u_{l},u_{t+4}\}\in E(\bar{G})$. 
\item[$\bullet$] If $l=1$, then $\bar{G}=G_{(a)_1}$ in Figure~\ref{type a}.  
\item[$\bullet$] If $l=2$, then $\bar{G}=G_{(a)_2}$ in Figure~\ref{type a}. 
\item[$\bullet$]  If  $3\leq l<t+3$, then there is a minimal cycle $C'=u_1-u_{t+4}- u_l- u_{l+1}-\cdots-u_{t+3}-u_1$ of length $t+6-l$. By Lemma~\ref{sedaghashang}, $t+6-l=t+3$ which implies $l=3$. If $\{u_2,u_{t+4}\}\notin E(\bar{G})$, then we will have a minimal $4$-cycle on the  vertex set $\{u_1,u_2,u_3,u_{t+4}\}$. It follows from  Lemma~\ref{sedaghashang} that $|C|=4$. Hence, $\bar{G}$ is isomorphic to $G_{(c)}$ in Figure~\ref{type c}. If $\{u_2,u_{t+4}\}\in E(\bar{G})$, then $\bar{G}=G_{(a)_3}$ in Figure~\ref{type a}.
\item[$\bullet$] If $l=t+3$, then since $G$ does not have isolated vertices, there exists $1<j<t+3$, such that $\{u_{t+4},u_j\}\notin E(\bar{G})$.  Let $k,k'$ with $1\leq k< j<k'\leq t+3$ be respectively the largest index and the smallest index such  that  $\{u_k,u_{t+4}\}, \{u_{k'},u_{t+4}\}\in E(\bar{G})$. It follows that  $C'=u_{t+4}-u_k- u_{k+1}-\cdots-u_{k'}-u_{t+4}$ is a minimal cycle and hence $|C'|=k'-k+2=t+3$. Therefore we have  either $(k,k')=(1,t+2)$ or $(k,k')=(2,t+3)$. In both cases $\bar{G}$ is isomorphic to $G_{(a)_3}$.

{\em Case} (ii): Suppose   $[n]\setminus V(C)=\{u_{t+4},u_{t+5}\}$.  By Corollary~\ref{ostad}(a) both $u_{t+4},u_{t+5}$ are adjacent to at least one vertex of $C$ in $\bar{G}$. 
	\item[$\bullet$] Suppose in the graph $\bar{G}$ the vertex $u_{t+4}$  is adjacent to  at most $2$ vertices of $C$ one of which is $u_1$. Then $u_{t+5}$ is adjacent to all vertices of $C$ in $\bar{G}$ by  Corollary~\ref{ostad}(b). Since $\Delta_W$ is connected for $W=[n]\setminus \{u_1\}$, we conclude that $u_{t+4}$ is adjacent to some vertex in $[n]\setminus \{u_1\}$ in $\bar{G}$ and since $u_{t+5}$ is not isolated in $G$, $u_{t+5}$ is not adjacent to $u_{t+4}$ in $\bar{G}$  and hence  $u_{t+4}$ is adjacent to some  $u_j\in V(C)$ with $j\neq 1$ in $\bar{G}$. We show that either $j=2$ or else $j=t+3$. Otherwise there is a minimal cycle $C':=u_{t+4}-u_1-u_2-\cdots-u_j-u_{t+4}$ of length $j+1$ which is equal to $t+3$, by Lemma~\ref{sedaghashang}. Thus $j=t+2$ which implies that $C'':=u_{t+4}-u_{t+2}-u_{t+3}-u_1-u_{t+4}$ is a minimal $4$-cycle and hence $t=1$. But   $T(C''')\in \ker\partial_1\setminus\im\partial_{2}$, where $C''':=u_{5}-u_3-u_{6}-u_1-u_{5}$ is a minimal $4$-cycle  in $\bar{G}$. It follows that $\beta_{3,6}(I)\neq 0$, a contradiction. Thus either $j=2$ or else $j=t+3$ and hence $\bar{G}$ is isomorphic to $G_{(b)}$ in Figure~\ref{type b}. Same holds if we interchange $u_{t+4}$ and $u_{t+5}$ in the above argument.

	\item[$\bullet$]  Now suppose $u_{t+4}, u_{t+5}$ are adjacent to at least $3$ vertices of $C$ in $\bar{G}$. 
	If   $u_{t+4}$ as well as $u_{t+5}$ is not adjacent to some vertices of $C$ in $\bar{G}$, then as seen in the argument of (O-4)(ii), the graph $\bar{G}$ is isomorphic to $G_{(d)_2}$ in Figure~\ref{type d}.  
	
	Now consider the case that at least one of the vertices $u_{t+4}, u_{t+5}$, say $u_{t+5}$, is adjacent to all vertices of $C$ in $\bar{G}$.  The argument below  also works if we interchange $u_{t+4}, u_{t+5}$.
	
	 Suppose  $u_{t+4}$ is adjacent to (at least) three vertices $u_1, u_k, u_j$ of $C$  with $1<k<j\leq t+3$ in the graph $\bar{G}$. Since $u_{t+5}$ is not isolated in $G$, we have $\{u_{t+4},u_{t+5}\}\notin E(\bar{G})$.  If $(k,j)\neq (2,t+3)$, then we get the minimal $4$-cycle $C'=u_{t+4}-u_{1}-u_{t+5}-u_l-u_{t+4}$, where $l=k$ if $k\neq 2$, and else $l=j$, and  hence $t=1$.  If  $(k,j)= (2,t+3)$, then  we  get the minimal $4$-cycle $C''=u_{t+4}-u_{2}-u_{t+5}-u_{t+3}-u_{t+4}$ and so  $t=1$ also in this case. 
From $t=1$ we conclude that $u_1, u_k, u_j$ are   successive vertices in $C$. 	
	Without loss of generality  we may assume that $(k,j)=(2,3)$.  Since  $\{u_5,u_6\}\notin E(\bar{G})$, in case $\{u_{5}, u_4\}\notin E(\bar{G})$,  
	   the graph  $\bar{G}$ is  isomorphic to $G_{(d)_2}$, and in case $\{u_{5}, u_4\}\in E(\bar{G})$, it  is isomorphic to $G_{(d)_1}$ in Figure~\ref{type d}.
 This completes the proof.
	\end{proof}

All the arguments so far in this section were characteristic independent; consequently 
\begin{cor}\label{final note}
	The property of having almost maximal finite index for edge ideals is independent of the characteristic of the base field. In other words, given a simple graph $G$,  its edge ideal $I(G)$ has almost maximal finite index over some field if and only if it has this property over all fields. 
\end{cor}
\begin{cor}\label{depth}
	Let $G$ be a simple graph on $[n]$ with no isolated vertex such that $I=I(G)\subset\!S$ has almost maximal finite index.   Then over all fields
	\begin{align*}
\mathrm{pd} (I)=\begin{cases} n-3 \quad \text{if } \bar{G}=G_{(c)} \text{ or } G_{(a)_i},\ i=1,2,3,\\ n-4 \quad \text{if } \bar{G}=G_{(b)}\text{ or } G_{(d)_i}, \ i=1,2.\end{cases}
\end{align*}
In particular, $3\leq \mathrm{depth} (I)\leq 4$.
\end{cor}
\begin{proof}
	Let $\mathrm{index}(I)=t$. By Theorem~\ref{check-out}, $\bar{G}\in\{G_{(a)_1}, G_{(a)_2}, G_{(a)_3}, G_{(b)}, G_{(c)}, G_{(d)_1}, G_{(d)_2}\}$.  
	It follows that 
	\begin{align*}
\hspace{1cm} n=\begin{cases} t+4\quad \text{if } \bar{G}=G_{(c)}, G_{(a)_i},\ i=1,2,3,\\ t+5\quad \text{if } \bar{G}=G_{(b)}, G_{(d)_i},\  i=1,2.\end{cases}
\end{align*}
Since $\projdim(I)=t+1$, the assertion follows. The last assertion follows from the Auslander-Buchsbaum formula. 
\end{proof}
In the rest of this section we study the last graded Betti numbers of edge ideals with almost maximal finite index. We first see in the following lemma  that the graded Betti numbers of the edge ideals with this property are  independent of the characteristic of the base field. 
The proof takes a great benefit of Katzman's paper  \cite{Ka}.

\begin{lem}\label{char 2}
Let $I\subset S$ be the edge ideal of a simple graph with almost maximal finite index. Then the Betti numbers of $I$ are characteristic independent. 
\end{lem}

\begin{proof}
\cite[Theorem~4.1]{Ka} states that the Betti numbers  of the edge ideals of the graphs with at most $10$ vertices are independent of the characteristic of the base field.  
 It follows that the graded Betti numbers of $I=I(G)$ with $\Index (I)=t$  are characteristic independent when $\bar{G}\in \{G_{(c)},G_{(d)_1}, G_{(d)_2}\}$. 

By \cite[Corollary~1.6, Lemma~3.2(b)]{Ka}, if $G$ has a vertex $v$ of degree $1$ or at least $|V(G)|-4$, then the Betti numbers of $I(G)$ are characteristic independent if and only if the Betti numbers of $I(G-\{v\})$ are characteristic independent. Here  $G-\{v\}$ is the induced subgraph of $G$ on $V(G)\setminus\{v\}$. Since the vertex $t+4$ is  of degree  one  in $\overline{G_{(b)}}$, it follows that the Betti numbers of $I(\overline{G_{(b)}})$ are characteristic independent if and only if  so are the Betti numbers of  $I(\overline{G_{(a)_2}})$.   For $\bar{G}\in \{G_{(a)_i}:\ 1\leq i\leq 3\}$, since $t+4$ is adjacent to at least $t$ vertices in  the graph $G$ and since  $|V(G)|=t+4$, it is enough to show that the Betti numbers of  $I(G-\{t+4\})$ are characteristic independent. But $G-\{t+4\}$ is the complement of a minimal cycle  of length $t+3$. Note that by Hochster's formula, all the linear Betti numbers $\beta_{i,i+2}(I)$ are obtained from computing the dimension of $\widetilde{H}_0(\Delta(\bar{G})_W;\mathbb{K})$  with  $W\subseteq V(G)$ and $|W|=i+2$, and this dimension equals  the number of connected components of $\bar{G}_W$ minus one. Therefore these Betti numbers do not depend on the characteristic of the base field, see also \cite[Corollary~1.2(b)]{Ka}. Moreover, as seen in \cite[Theorem~4.1, Proposition~4.3]{BHZ},  the edge ideal of the complement of a minimal cycle has one nonzero non-linear Betti number $\beta_{t,t+3}(I)=1$ over all fields. Therefore the Betti numbers of  $I(G-\{t+4\})$ are characteristic independent, as desired.
\end{proof}

For the edge ideals with linear resolution all non-linear Betti numbers are zero. For the edge ideals  with maximal finite index $t$, it is seen in \cite{EGHP, BHZ} that there is only one nonzero non-linear Betti number $\beta_{t,t+3}(I)=1$ over all fields.   In the  case of ideals with almost maximal finite index  with $\Index(I)=t$, the  non-linear Betti numbers  appear in the last two homological degrees of the minimal free resolution. By the arguments that we had so far, it is easy to compute the  $(t+1)$-th graded Betti numbers and also $t$-th non-linear Betti numbers, where $I$ is the edge ideal with almost maximal finite index. Nevertheless, in the cases $\bar{G}=G_{(c)}$ and $\bar{G}=G_{(d)_i}$ for $i=1,2$  one can see the whole Betti table, using {\em Macaulay 2}, \cite{M2}.  
Note that since all the graphs in Example~\ref{mesal} have at most $t+5$ vertices,  where $\Index(I(G))=t$, and since the edge ideals are generated in degree $2$, by Hochster's formula  it is enough to consider $\beta_{i,j}(I(G))$ for $ i+2\leq j\leq t+5$.

\begin{prop}\label{Bettis of almost}
 Let $G$ be a graph such that $I:=I(G)$ has almost maximal finite index $t$. Then over all fields, $\beta_{t,t+4}(I)=\beta_{t,t+5}(I)=0$ and 
  \begin{align*}
  \beta_{t,t+3}(I)=\begin{cases} 1\quad \text{if } \bar{G}=G_{(a)_1} \text{ or } G_{(a)_2}\text{ or } G_{(b)},\\ 2\quad\text{if } \bar{G}=G_{(a)_3},\\ 3\quad\text{otherwise},\end{cases} 
  \end{align*}
  \begin{align*}
\beta_{t+1,t+3}(I)&=\begin{cases} 1\quad \text{if } \bar{G}=G_{(a)_1} \text{ or } G_{(b)},\\ 0\quad\text{otherwise,}\end{cases}\\
 \beta_{t+1,t+4}(I)&=\begin{cases} 2\quad \text{if } \bar{G}=G_{(c)} \text{ or } G_{(d)_2},\\ 0\quad\text{if }\bar{G}=G_{(d)_1},\\ 1\quad\text{otherwise,}\end{cases}\\
  \beta_{t+1,t+5}(I)&=\begin{cases} 1\quad \text{if } \bar{G}=G_{(d)_1},\\ 0\quad\text{otherwise.}\end{cases}
\end{align*}
In particular, 
\begin{align*}
\hspace*{-2.cm}\mathrm{reg} (I)=\begin{cases} 4\quad \text{if } \bar{G}=G_{(d)_1},\\ 3\quad\text{otherwise.}\end{cases}
\end{align*}
 \end{prop}
 \begin{proof}
 All the equalities are straightforward consequences of the use of Hochster's formula and Observation~\ref{connected induced graphs}. However, the  Betti number $\beta_{t,t+3}(I)$  can be also deduced from \cite[Theorem~4.6]{FG}.  
  It is worth to emphasize  that although $\widetilde{H}_1(\Delta,\mathbb{K})$ is spanned by the set  of $0\neq T(C)+\im\partial_2$ for all minimal cycles $C$ of $\bar{G}$, this set  may not be  a basis. In case $\bar{G}=G_{(c)}$, the graph $\bar{G}$ has three minimal cycles $C$ of length $4$ with $T(C)\notin \im\partial_{2}$,  but for the cycle $C=1-2-3-4-1$, $T(C)$ is a  linear combination of   $T(C'), T(C'')$, where $C', C''$ are the two other cycles of $G_{(c)}$. Hence $\dim_{\mathbb{K}}\widetilde{H}_1(\Delta(G_{(c)}),\mathbb{K})=2$.  In the case of   $G_{(a)_3}$, we have  $0\neq T(C)+ \im\partial_{2}$, where  $C$ is the minimal cycle on $[t+3]$, but $T(C)$ is a linear combination of $T(C'), T(C''),T(C''')$, where $C'$ is the minimal cycle on $[t+4]\setminus \{2\}$, and $C'', C'''$ are the two triangles in $G_{(a)_3}$ and hence $\dim_\mathbb{K}\widetilde{H}_1(\Delta(G_{(a)_3}),\mathbb{K})=1$.
 \end{proof}
 

\section{Powers of edge ideals with large Index}\label{powers of almost maximal}
Due to a result of Herzog, Hibi and Zheng, \cite[Theorem~3.2]{HHZh1}, if the edge ideal $I:=I(G)$ has a linear resolution, that is $\Index (I)=\infty$, then all of its powers have a linear resolution as well. In case $I$ has maximal finite index $t>1$, then by \cite[Corollary~4.4]{BHZ}  the ideal $I^s$ has a linear resolution for all $s\geq 2$. Note that in general for any edge ideal $I$ with  $\Index(I)=1$, one has  $\Index(I^s)=1$ for all $s\geq 2$, see Remark~\ref{part 2 of theorem} below. In this section we investigate when the higher powers of the edge ideal $I$ with  almost maximal finite index have  a linear resolution.   Indeed, the aim of this section is to prove the following:

\begin{thm}\label{powers}
	Let $G$ be a simple graph  with no isolated vertex whose edge ideal $I(G)\subset S$ has almost maximal finite index. Then $I(G)^s$ has a linear resolution for all $s\geq 2$ if and only if $G$ is gap-free.  
\end{thm}

{ Theorem~\ref{powers} follows from Remarks~\ref{part 2 of theorem} and \ref{part 1 of theorem}, and Theorems~\ref{main G_a3} and \ref{I^k has lin res}. Indeed,} we will see in Remark~\ref{part 1 of theorem} that this theorem holds for $G$ with $G\in\{\overline{G_{(a)_1}}, \overline{G_{(a)_2}}\}$ with $t>1$. For $G=\overline{G_{(a)_3}}$ with $t>1$ and $G=\overline{G_{(b)}}$ we will prove the assertion in Theorem~\ref{main G_a3} and Theorem~\ref{I^k has lin res}, respectively. As it is mentioned in Remark~\ref{part 2 of theorem} below,  all other  graphs whose edge ideals have almost maximal finite index contain a gap. 
Recall that a {\em gap} in a graph $G$ is an induced subgraph on $4$ vertices and a pair of edges with no vertices in common which are not linked by a third edge; see the graph $G_1$ in Figure~\ref{gcd}. The graph $G$ is called {\em gap-free} if it does not admit a gap; equivalently if $\bar{G}$ does not contain an induced $4$-cycle. This property plays an important role in the study of the resolution of  powers of edge ideals; for example

\begin{prop}\label{gap free}{$($Francisco-H\`a-Van Tuyl;  unpublished, see  \cite[Proposition~1.8]{NP} and  \cite[Theorem~3.1]{BHZ}$)$}
	Let $G$ be a simple graph. If $I(G)^s$ has a linear resolution for some $s\geq 1$, then $G$ is gap-free.
\end{prop}

On the other hand, 

\begin{rem}\rm \label{part 2 of theorem}
	A more precise statement about the gap-free graphs is given in  \cite[Theorem~3.1]{BHZ} which  says that  for a  graph $G$ the following are equivalent:
	\begin{itemize}
		\item[(a)] $G$ admits a gap;
			\item[(b)] $\Index(I(G)^s)=1$ for all $s\geq 1$;
			\item [(c)] there exists $s\geq 1$ with $\Index(I(G)^s)=1$.
				\end{itemize}
	If $G$ is the graph whose complement is one of $G_{(a)_1}, G_{(a)_2}, G_{(a)_3}, G_{(b)}$ with   $t=1$, or one of  $G_{(c)}, G_{(d)_1}, G_{(d)_2}$, then  $G$ has a gap. So by the above equivalence $\Index(I(G)^s)=1$ for all $s\geq 1$ in this case.  
	\end{rem}\rm  

In order to prove Theorem~\ref{powers} for $G=\overline{G_{(a)_3}}$ with $t>1$, we need the following result of Banerjee~\cite{Ba}.


\begin{thm} \cite[Theorem~5.2]{Ba}\label{Banerjee1}
Let $G$ be a simple graph  and let $I:=I(G)$ be its edge ideal. Let $\mathcal{G}(I^s)=\{m_1,\ldots, m_r\}$. Then for all $s\geq1$
$$\mathrm{reg}(I^{s+1})\leq\max\{\mathrm{reg}(I^s),\ \mathrm{reg}(I^{s+1}: m_k) + 2s \text{ for }  1\leq k\leq r \},$$
where $(I^{s+1}: m_k)$ denotes the colon ideal, i.e., $(I^{s+1}: m_k)=\{f\in S:\ fm_k\in I^{s+1}\}$.
\end{thm}

As a consequence  of this theorem, Banerjee showed in \cite[Theorem~6.17]{Ba} that for any gap-free and cricket-free graph $G$, the ideal $I(G)^s$ has a linear resolution for  all $s\geq 2$. A {\em cricket} is a graph isomorphic to the graph $G_2$ in Figure~\ref{gcd}, and a graph $G$ is called {\em cricket-free} if $G$  contains no cricket as an induced subgraph. 

Two other classes of graphs which produce edge ideals whose higher powers have linear resolution  were given by Erey. She proved in \cite{Er, Er1} that  $I(G)^s$ has a linear resolution for  all $s\geq 2$ if $G$ is   gap-free and also it is either diamond-free or $C_4$-free. A {\em diamond} is a graph isomorphic to the graph $G_3$ in Figure~\ref{gcd}, and a {\em diamond-free} graph is a graph with no diamond as its induced subgraph. A $C_4$-free graph is a graph which does not contain a $4$-cycle as an induced subgraph; i.e. its complement is gap-free.
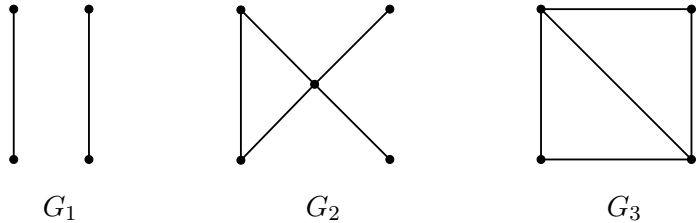
\begin{figure}[ht!]
\hspace{.1cm}
	\begin{tikzpicture}[line cap=round,line join=round,>=triangle 45,x=.5cm,y=.5cm]
	\clip(2.7,5.3) rectangle (22,11.5);
	\draw [line width=.7pt] (3,11)-- (3,7);
	\draw [line width=.7pt] (5,11)-- (5,7);
	\draw [line width=.7pt] (9.04,10.98)-- (9.04,6.98);
	\draw [line width=.7pt] (11,9)-- (9.04,6.98);
	\draw [line width=.7pt] (11,9)-- (9.04,10.98);
	\draw [line width=.7pt] (11,9)-- (13,11);
	\draw [line width=.7pt] (11,9)-- (13,7);
	\draw [line width=.7pt] (17.02,11)-- (17.02,7);
	\draw [line width=.7pt] (21.02,7)-- (17.02,7);
	\draw [line width=.7pt] (21.02,11)-- (21.02,7);
	\draw [line width=.7pt] (21.02,11)-- (17.02,11);
	\draw [line width=.7pt] (17.02,11)-- (21.02,7);
	\draw (3.5,6.3) node[anchor=north west] {$G_1$};
	\draw (10.5,6.3) node[anchor=north west] {$G_2$};
	\draw (18.5,6.3) node[anchor=north west] {$G_3$};
	\begin{scriptsize}
	\draw [fill=black] (3,11) circle (1.5pt);
	\draw [fill=black] (3,7) circle (1.5pt);
	\draw [fill=black] (5,11) circle (1.5pt);
	\draw [fill=black] (5,7) circle (1.5pt);
	\draw [fill=black] (9.04,10.98) circle (1.5pt);
	\draw [fill=black] (9.04,6.98) circle (1.5pt);
	\draw [fill=black] (11,9) circle (1.5pt);
	\draw [fill=black] (13,11) circle (1.5pt);
	\draw [fill=black] (13,7) circle (1.5pt);
	\draw [fill=black] (17.02,11) circle (1.5pt);
	\draw [fill=black] (17.02,7) circle (1.5pt);
	\draw [fill=black] (21.02,7) circle (1.5pt);
	\draw [fill=black] (21.02,11) circle (1.5pt);
	\end{scriptsize}
	\end{tikzpicture}
	\caption{$G_1$ a gap, $G_2$ a cricket, $G_3$ a diamond}
	\label{gcd}
\end{figure}
\begin{rem}\rm \label{part 1 of theorem}
	Clearly, the graphs $\overline{G_{(a)_1}}, \overline{G_{(a)_2}}$ are cricket-free and hence the statement of Theorem~\ref{powers} holds in these two cases using  \cite[Theorem~6.17]{Ba}. Note that these graphs  are gap-free for $t\geq 2$.

	 On the other hand, the graphs  $\overline{G_{(a)_3}}$ and $\overline{G_{(b)}}$ contain crickets for large enough $t$.  Indeed, if $t\geq 3$, then the induced subgraph of  $\overline{G_{(a)_3}}$ on the vertex set $\{1,2,3, 5,t+4\}$, and if $t\geq 2$, then   the induced subgraph of  $\overline{G_{(b)}}$ on  $\{3,4,5,t+4,t+5\}$ are isomorphic to a cricket. These  graphs are not diamond-free in general as well, because for $t\geq 3$, the induced subgraphs on the vertex sets $\{2,4,6,t+4\}$ and $\{3,5,6,t+5\}$ form respectively diamonds in  $\overline{G_{(a)_3}}$ and $\overline{G_{(b)}}$. They are not even $C_4$-free for $t\geq 3$ since ${G_{(a)_3}}$ and ${G_{(b)}}$ contain gaps. 
	 Therefore, when $G\in \{\overline{G_{(a)_3}},\overline{G_{(b)}}\}$ and $t$ is large enough, one cannot take benefit of the results of Banerjee or Erey to deduce Theorem~\ref{powers}. 
	\end{rem}\rm

It is  shown in \cite[Section~6]{Ba} that for the edge ideal $I$ of a simple graph $G$ and the minimal generator $m_k$ of $I^s$, $s\geq1$, the ideal $(I^{s+1}: m_k)$ is a quadratic monomial ideal whose polarization coincides with the edge ideal of a simple graph  with the construction explained in Lemma~\ref{Banerjee3} below. For the details about the  polarization technique, the reader may consult with \cite{HHBook}. Throughout this section, for an edge $e=\{i,j\}$ of a graph $G$ its associated quadratic monomial $x_ix_j$ is   denoted by  ${\bf x}_e$. 
\begin{lem} \cite[Lemma~6.11]{Ba}\label{Banerjee3}
Let $G$ be a simple graph with the edge ideal  $I:=I(G)$, and let $m_k={\bf x}_{e_1}\cdots {\bf x}_{e_s}$ be a minimal generator of $I^s$, where   $e_1,\ldots, e_s$ are some edges of $G$. Then the polarization $(I^{s+1}: m_k)^{pol}$ of the ideal $(I^{s+1}: m_k)$   is the edge ideal of a new graph $G_{e_1\ldots e_s}$ with the following structure:
\begin{itemize}
\item[(1)] $V(G)\subseteq V(G_{e_1\ldots e_s})$, $E(G)\subseteq E(G_{e_1\ldots e_s})$. 
\item[(2)] Any two vertices $u, v$, $u\neq v$, of $G$ that are even-connected with respect to $e_1\cdots e_s$ are connected by an edge in $G_{e_1\ldots e_s}$. 
\item[(3)] For every vertex $u$ which is even-connected to itself with respect to $e_1\cdots e_s$ there is a new vertex $u'\notin V(G)$ which is connected to $u$  in $G_{e_1\ldots e_s}$ by an edge  and not connected to any other vertex $($so $\{u,u'\}$ is a whisker in $G_{e_1\ldots e_s}$$)$.
\end{itemize}
\end{lem}

In \cite{Ba}, two vertices $u$ and $v$  of a graph $G$ ($u$ may be same as $v$) are said to be {\em even-connected} with respect to an $s$-fold product $e_1\cdots e_s$ in $G$ if there is a path $P=p_0-p_1-\cdots-p_{2k+1}$, $k\geq 1$, in $G$ such that:
\begin{itemize}
\item[(1)] $p_0=u, p_{2k+1}=v.$
\item[(2)] For all $0\leq l \leq k-1$, $\{p_{2l+1}, p_{2l+2}\}=e_i$ for some $1\leq i\leq s$.
\item[(3)] For all $i$, $|\{l:\ 0\leq l\leq k-1, \ \{p_{2l+1}, p_{2l+2} \}=e_i \} | \leq | \{j :\ 1\leq j\leq s, \ e_j=e_i \} |$.
\item[(4)] For all $0 \leq r \leq 2k$, $\{p_r, p_{r+1}\}$ is an edge in $G$.
\end{itemize}

\medskip

	Now we are ready to prove Theorem~\ref{powers} for $G=\overline{G_{(a)_3}}$ with $t>1$.  
	\begin{thm}\label{main G_a3}
		Let $G$ be a   graph on $n\geq 6$ vertices such that  $G_{(a)_3}$ is its complement. Let $I:=I(G)$ be the edge ideal of $G$. Then $I^s$ has a linear resolution for $s\geq 2$.
	\end{thm}
	
	\begin{proof}
		Note that $t+4=n\geq 6$ implies that $t>1$. We first show  that for any 
		$s\geq 1$	and any 	$s$-fold product $e_1\cdots e_s$ of the edges in $G$, the  graph $\overline{G_{e_1\cdots e_s}}$ is chordal, where $G_{e_1\cdots e_s}$ is a simple graph explained in Lemma~\ref{Banerjee3} with the edge ideal $I(G_{e_1\cdots e_s})=(I^{s+1}:\ {\bf x}_{e_1}\cdots {\bf x}_{e_s})^{pol}$. 
		
		Since by \cite[Lemmas~6.14, 6.15]{Ba}, any induced cycle of $\overline{G_{e_1\cdots e_s}}$ is an induced cycle of $\bar{G}$, we conclude that if $\overline{G_{e_1\cdots e_s}}$ contains an induced cycle $C$ of length $>3$,  then  $C\in \{C_1, C_2\}$, where $C_1=1-2-\cdots-(t+3)-1$ and $C_2=1-(t+4)-3-4-\cdots-(t+3)-1$. Thus, in order to prove that $\overline{G_{e_1\cdots e_s}}$ is  chordal, we need to show that 
		$C_1,C_2$ are not induced cycles in $\overline{G_{e_1\cdots e_s}}$  for $s\geq 1$.
		
		We claim that there exist $k,l\in V(G_{e_1\cdots e_s})$ such that $\{k,l\}\in E(G_{e_1\cdots e_s})\cap E(C_r)$, $r=1,2$. It follows that $C_r$ is not a subgraph of $\overline{G_{e_1\cdots e_s}}$, as desired.
		
		\medspace
		\textit {Proof of the claim:}  
		Let $e_1=\{i,j\}$ with $i<j$ and let $s\geq 1$.  We choose  $\{k,l\}\in  E(C_1)$ as follows:
		\begin{itemize}
			\item[$(a)$] If $e_1=\{4,t+4\}$, then let $k=1$ and $l=t+3$;
			\item[$(b)$] if   $e_1=\{1,t+2\}$, then let $k=3$ and $l= 2$;
			\item[$(c)$]  if $e_1=\{2,t+3\}$, then let $k=4$ and $l= 3$;
			\item[$(d)$] otherwise, let $k=\overline{i-2}$ and $l=\overline{i-1}$.
		\end{itemize}
		Since $C_2$ is obtained from $C_1$ by replacing  $2$ with $t+4$, in order to find  $\{k,l\}\in E(C_2)$, we choose $\{k,l\}\in E(C_2)$ as suggested  in $(a)-(d)$  with an extra condition that if $\{k,l\}$ is obtained from $(b),(d)$ and it contains $2$,   then  we replace $2$ with $t+4$ in this pair.

		\medspace
		By the above choices of $k,l$, although $\{k,l\}\notin E(G)$, we have $\{k,i\}, \{j, l\}\in E(G)$.  It follows that $k-i-j-l$ is a path in $G$ and hence, by definition,  $k$ and $l$ are even-connected with respect to $e_1\cdots e_s$. Therefore $\{k,l\}\in E(G_{e_1\cdots e_s})$.   
		This completes the~proof~of~the~claim.
		
		\medspace
		Now since  $\overline{G_{e_1\cdots e_s}}$ is chordal for $s\geq 1$, by \cite[Theorem~1]{Fr}, $I(G_{e_1\cdots e_s})$ has a $2$-linear resolution for $s\geq 1$. It follows that for any choice of the edges  $e_1,\ldots, e_s$ of $G$ one has 
		\begin{align*}\mathrm{reg}((I^{s+1}:\ {\bf x}_{e_1}\cdots {\bf x}_{e_s}))=\mathrm{reg}((I^{s+1}:\ {\bf x}_{e_1}\cdots {\bf x}_{e_s})^{pol})=\mathrm{reg}(I(G_{e_1\cdots e_s}))=2.
		\end{align*}
		The first equality follows from \cite[Corollary~1.6.3]{HHBook}. 
		By Proposition~\ref{Bettis of almost} we have $\mathrm{reg}(I)=3$.  Theorem~\ref{Banerjee1} implies that  $\mathrm{reg}(I^2)\leq 4$. Since $I^2$ is generated in degree $4$ we conclude that $\mathrm{reg}(I^2)=4$. 
		Now induction on $s>1$ and using Theorem~\ref{Banerjee1} yield the assertion. 
	\end{proof}

\medskip
Now it  remains to prove Theorem~\ref{powers} for $\overline{G_{(b)}}$.  The crucial point in the proof of Theorem~\ref{powers} for $\overline{G_{(a)_3}}$ was  to show that $\mathrm{reg}((I(\overline{G_{(a)_3}})^{s+1}:\ m_k))=2$ for all minimal generators $m_k$ of $I(\overline{G_{(a)_3}})^{s}$. Having proved this statement, we deduced that  the upper bound of $\mathrm{reg}(I(\overline{G_{(a)_3}})^{s+1})$ in   Theorem~\ref{Banerjee1} is  $2s+2$ and hence the desired conclusion was followed. 
The same  method will not work for $\overline{G_{(b)}}$. Indeed, 
  computations by  {\em Macaulay~2}, \cite{M2}, shows that  $\mathrm{reg}((I(\overline{G_{(b)}})^{s+1}:\ m_k))=3$, where $s\geq 1$ and $m_k=x_{t+5}^sx_{t+4}^s$ is a minimal generator of $I(\overline{G_{(b)}})^{s}$. Hence the upper bound of $\mathrm{reg}(I(\overline{G_{(b)}})^{s+1})$ in   Theorem~\ref{Banerjee1} is at least $2s+3$ which is greater than the degree of the generators of $I(\overline{G_{(b)}})^{s+1}$ and consequently one cannot deduce that this ideal has a linear resolution only by computing the upper bound in Theorem~\ref{Banerjee1}. Therefore 
  in order to prove Theorem~\ref{powers} 	 for $\overline{G_{(b)}}$ we  need some other tool. This tool is provided in the following result of Dao et al.

\begin{lem}\label{Dao}{\cite[Lemma~2.10]{DHS}}
	Let $I\subset S$ be a monomial ideal, and let $x$ be a variable appearing in some generator of $I$.  Then $$\mathrm{reg}(I)\leq \max\{\mathrm{reg}((I:x)) + 1,\mathrm{reg}(I+(x))\}.$$ Moreover, if $I$ is squarefree, then $\mathrm{reg}(I)$ is equal to one of these terms.
\end{lem}

We will apply this result for $I:=I(\overline{G_{(b)}})^{s+1}$, $s\geq 1$, and $x:=x_{t+5}$. In
 Theorem~\ref{linquo of square} we will compute the regularity of the ideal $(I(\overline{G_{(b)}})^{s+1}:x_{t+5})$, $s\geq 1$,  by showing that it has linear quotients.
 Recall that
a graded ideal $I$ is said to have {\em linear quotients} if there exists a  homogeneous  generating set of $I$, say $\{f_1,\ldots,f_m\}$, such that the colon ideal $\left((f_1,\ldots,f_{i-1}):f_i\right)$ is generated by variables for all $i>1$. By \cite[Theorem~8.2.1]{HHBook} equigenerated ideals with linear quotients have a linear resolution.

In the next result,  Proposition~\ref{step1}, we provide a step of the proof of  Theorem~\ref{linquo of square} which is a bit long yet easy to follow. In the proof of this proposition  and also Theorem~\ref{linquo of square} we need to order the generators of the given ideals. To this end we should first order the multisets of edges of the associated graphs. We will use the following order in both proofs:

\medspace
{\em	 Let $G$ be a simple graph. For $e=\{i,j\}\in E(G)$ with $i<j$ and $e'=\{i',j'\}$ with $i'<j'$, we let $e<e'$ if either $i<i'$ or  $i=i'$ with $j<j'$.  	
	Let $r\geq 1$. We denote by 	$\mathbf{e}:=(e_{i_1},\ldots,e_{i_{r}})$ the multiset $\{e_{i_1},\ldots,e_{i_{r}}\}$ of the edges of $G$, where $e_{i_1}\leq\cdots\leq e_{i_{r}}$. If $\mathbf{e'}:=(e_{i'_1},\ldots,e_{i'_{r}})$ is another ordered multiset in $E(G)$ of the same size $r$, we let 
	$\mathbf{e}\leq \mathbf{e'}$  if either $\mathbf{e}=\mathbf{e'}$ or else there exists $1\leq j\leq r$ such that $e_{i_l}=e_{i'_l}$ for all $l<j$ and $e_{i_j}<e_{i'_j}$.}
	
	\medspace
	For the ordered multiset $\mathbf{e}:=(e_{i_1},\ldots,e_{i_{r}})$ of the edges of the graph $G$, we denote by $\x{\mathbf{e}}$ the monomial $\x{e_{1}}\cdots\x{e_{r}}$.
	Moreover,  we denote by $\supp(m)$  the set of all variables dividing the monomial $m\in S$ and also  denote by $\deg_{m}x_{i}$   the largest integer $d$ such that  $x_i^d$ divides $m$. We use the notation $m|m'$ ($m\centernot|m'$ resp.) when a monomial $m$ divides (does not divide resp.) a monomial $m'$.  
	
\begin{prop}\label{step1}
Let $C=1-2-\cdots-(t+3)-1$, $t\geq 1$, be a cycle graph and let $t+4$ be a vertex not belonging to $C$. Then for $s\geq 0$ the ideal $L=I(\bar{C})^{s}(x_3,\ldots,x_{t+4})$ has linear quotients. 
\end{prop}
\begin{proof} 
For $s=0$ the assertion is obvious. Suppose $s\geq 1$. 
We order the edges of $G:=\bar{C}$ as described above. 
	Each element $m$ in the minimal generating set $\mathcal{G}(L)$ of $L$ can be written as  $m=\x{\e{}}x_k$,
	 where $\e{}=(e_{1},\ldots, e_{{s}})$ is an ordered multiset of the edges of $\bar{C}$ and $x_{k}\in \{x_3,\ldots,x_{t+4}\}$.  
	Note that there may be different multisets associated to $m$ and hence different presentations of $m$ as above. 
	In this case we consider the presentation of $m$ whose associated ordered multiset of the edges is the smallest. This means that if there is another presentation  of $m$ as $\x{\e{}'}x_{k'}$ with $\e{}\neq \e'{}$, then   we consider the presentation $\x{\e{}}x_k$ for $m$ if $\e{}<\e{}'$. By this setting, each minimal generator $m_l$ of $L$ has a unique {\em smallest} presentation $m_l=\x{\e{l}}x_k$, where $\e{l}$ denotes the smallest multiset of the edges associated to $m_l$.  
	
Now we order the generators of $L$ as follows: for $m_q,m_l\in\mathcal{G}(L)$ with $m_q=\x{\e{q}}x_{k'}$, $m_l=\x{\e{l}}x_{k}$, we let $m_q<m_l$ if either $\mathbf{e}_q<\mathbf{e}_{l}$ or $\mathbf{e}_q=\mathbf{e}_{l}$ with  $k'<k$. 

Suppose $\mathcal{G}(L)=\{m_1,\ldots,m_r\}$ with $m_1<\!\cdots\!<m_r$. We show that for any $m_l\in\!\mathcal{G}(L)$ with $l\!>1$, the ideal $\left((m_1,\ldots, m_{l-1}):m_l\right)$ is generated by some variables. Set $J_l:=(m_1,\ldots, m_{l-1})$. By \cite[Proposition~1.2.2]{HHBook}, the ideal $(J_l:m_l)$ is generated by the elements of the  set $\{m_q/\gcd(m_q,m_l):\ 1\leq q\leq l-1\}$. Let $m_{q,l}:=m_q/\gcd(m_q,m_l)$ for $m_q<m_l$. Suppose $m_l=\mathbf{x}_{\mathbf{e}_l} x_k$, $m_q=\mathbf{x}_{\mathbf{e}_q} x_{k'}$ with $\mathbf{e}_l:=(e_{1},\ldots,e_{s})$,  $\mathbf{e}_q:=(e'_{1},\ldots,e'_{s})$ and $3\leq k, k'\leq t+4$. Let $e_{i}=\{a_i, b_i\}$, $e'_{i}=\{a'_i, b'_i\}$ with $1\leq a_i<b_i-1\leq t+2$  and $1\leq a'_i<b'_i-1\leq t+2$ for $1\leq i\leq s$. 

In order to show that $(J_l:m_l)$ is generated by variables, we show that for each $q<l$, there exists $p<l$ such that $m_{p,l}$ is of degree one and it divides $m_{q,l}$. If $\deg m_{q,l}=1$, then we set $p:=q$ and so we are done. Assume that $\deg m_{q,l}>1$. First suppose $q=1$. Then $m_q=x_1^sx_3^{s+1}$. If $x_1| m_{q,l}$, then there exists $1\leq i\leq s$ with $1\notin e_i=\{a_i,b_i\}$. If $b_i\neq t+3$, 
then set $e:=\{1,b_i\}$ and if $b_i=t+3$ with $a_i\neq 2$, set  $e:=\{1,a_i\}$. Now set  $m_p:=(\x{\e {l}}\x{e}/\x{e_i})x_k$.
 Since $e<e_i$ we have  $p<l$. Moreover, $m_{p,l}=x_1$ and so we are done in this case.   Suppose $e_i=\{2,t+3\}$ for all $e_i\in \e{l}$ with $1\notin e_i$. It follows that $x_3|m_{q,l}$. If $k\neq 3$, then set $m_p:=\x{\e{l}}x_{3}$. Since $3<k$ we have $p<l$ and $m_{p,l}=x_3$.  If $k=3$, then set $e:=\{1,3\}$  and $m_p:=(\x{\e {l}}\x{e}/\x{e_i})x_{t+3}$.  
Now suppose $x_1\centernot| m_{q,l}$. Then $x_3|m_{q,l}$ and $1\in e_i$ for all $1\leq i\leq s$, and since $\deg m_{q,l}>1$ there exists $e_i\in\e{l}$ with $3\notin e_i$. Set $e:=\{1,3\}$  and $m_p:=(\x{\e {l}}\x{e}/\x{e_i})x_k$.  So we are done if $q=1$. Now suppose $q>1$ and for all $q'<q$ there is $p'<l$ with $\deg m_{p',l}=1$ and $m_{p',l}|m_{q',l}$. We prove the assertion by induction on $q$. 

Suppose there exist $e'_i\in \mathbf{e}_q$ and $e_j\in\mathbf{e}_l$  with $e'_i=e_j$. 
The monomials  $m'_{l}:= m_l/\mathbf{x}_{e_j}, m'_q:=m_q/\mathbf{x}_{e'_i}$ belong to  $I(\bar{C})^{s-1}(x_3,\ldots,x_{t+4})$ and $m'_q<m'_l$ and  $m_{q,l}=m'_q/\gcd(m'_q,m'_l)$.  If there exists $m'_p\in I(\bar{C})^{s-1}(x_3,\ldots,x_{t+4})$ with $m'_p<m'_l$, where $m'_p/\gcd(m'_p,m'_l)$ is of degree one  dividing $m'_q/\gcd(m'_q,m'_l)$, then setting $m_p:=m'_p\mathbf{x}_{e_i}$ one has $m_p\in J_l$ and  $\deg m_{p,l}=1$, where $m_{p,l}$ divides $m_{q,l}$, as desired. So it is enough to prove the assertion for $m'_q,m'_l$. Consequently, from now on  we may suppose that $\mathbf{e}_q\cap \mathbf{e}_l=\emptyset$.  In particular, $\mathbf{e}_q\neq\mathbf{e}_l$ and hence  
 $\mathbf{e}_q<\mathbf{e}_l$. Since $\mathbf{e}_q, \mathbf{e}_l$ do not share an  edge, it follows that $e'_1<e_1$ which means that either $a'_1<a_1$ or $a'_1=a_1$ with $b'_1<b_1$. 

{\em Case} (i): $a'_1<a_1$. If $a'_1=k$, then $3\leq k<a_1<b_1-1$ implies that $e:=\{k, b_1\}\in E(\bar{C})$ with $e<e_1$ and hence by interchanging 
$x_{a_1}$ in $\x{e_{1}}$ and $x_k$ we get 
a smaller presentation for $m_l$, a contradiction.  
Therefore $a'_1\neq k$. Note that $a'_1<a_1\leq a_i<b_i$ for all $i$. Thus $x_{a'_1}\centernot|m_l$ which implies that $x_{a'_1}|m_{q,l}$. 

Since $a'_1<b_i-1$ for all $i$, we have $e:=\{a'_1,b_i\}\in E(\bar{C})$, unless $\{a'_1,b_i\}=\{1,t+3\}$. If $\{a'_1,b_i\}=\{1,t+3\}$ for  all $i$ and if there exists $i$ with $a_i\neq 2$, then set $e:=\{a'_1,a_i\}$.  
 In both cases we have  $e<e_i$ and hence $m_p:= (\x{\e {l}}\x{e}/\x{e_i})x_k<m_l$ with $m_{p,l}=x_{a'_1}$. Suppose $a'_1=1$ and $e_i=\{2,t+3\}$ for  all $i$. We have $k\in\{3,t+3,t+4\}$, because otherwise by interchanging $x_{t+3}$ in $\x{e_i}$ and $x_k$ we get  a smaller presentation for $m_l$, a contradiction.  If $k=3$, then set $e:=\{1,3\}$ and $m_p:=(\x{\e {l}}\x{e}/\x{e_i})x_{t+3}$.  If $k\in\{t+3,t+4\}$, since $b'_1\notin\{2,t+3,t+4\}$, we have $x_{b'_1}\centernot|m_l$ and hence $x_{b'_1}|m_{q,l}$. Set $m_p:=\x{\e{l}}x_{b'_1}$.

\medspace
{\em Case} (ii): $a'_1=a_1$ and $b'_1<b_1$. If $x_{b'_1}|m_{q,l}$ then set  $m_p:=(\x{\e {l}}\x{e'_1}/\x{e_1})x_k$. Suppose $x_{b'_1}\centernot|m_{q,l}$ and hence $x_{b'_1}|m_l$. If $k=b'_1$, then interchanging $x_{b_1}$ in $\x{e_1}$ and $x_k=x_{b'_1}$ will result in  a smaller presentation for $m_l$, a contradiction.  Therefore,   $k\neq b'_1$ and hence $b'_1\in e_h$ for some $e_h\in\e{l}$. Thus $e_h\neq e_1$. If $b$ is another vertex of $e_h$, then $b\in\{b_1,b_1-1,\overline{b_1+1}\}$ since otherwise we get a smaller presentation of $m_l$ by interchanging $b_1$ in $e_1$ and $b'_1$ in $e_h$, a contradiction. 

Since $\deg m_{q,l}>1$, we have $\supp(m_{q,l})\neq \{x_{t+4}\}$. Thus $x_a|m_{q,l}$ for some  $a\neq t+4$.  If  $b_1\notin \{a,\overline{a-1},\overline{a+1}\}$ ($b\notin \{a,\overline{a-1},\overline{a+1}\}$ resp.), then set $e:=\{a,b_1\}$ ($e:=\{a,b\}$ resp.) and    $m_p:=(\x{\e {l}}\x{e'_1}\x{e}/(\x{e_1}\x{e_h}))x_k$. Suppose  
$b_1, b\in \{a,\overline{a-1},\overline{a+1}\}$ for all $a\neq t+4$ with $x_a|m_{q,l}$. 

 If $b_1= \overline{a+1}$, since $b_1>1$ we have $a\neq t+3$ and   $b_1=a+1$. Since $a_1<b'_1-1<b_1-1=a$  one has $e:=\{a_1,a\}\in E(\bar{C})$. Set $m_p:=(\x{\e{l}}\x{e}/\x{e_1})x_k$.  Now suppose $b_1\in \{a,\overline{a-1}\}$ for all $a$ with $x_a|m_{q,l}$ and $a\neq t+4$. 
 If $a=1$, then since $a'_1=a_1$ is the smallest vertex in $\e{q}$ one has $a_1=1$ and $b_1\in\{1,t+3\}$ which is a contradiction. If $a\in\{2,3\}$, then $b_1\in \{1,2,3\}$ which is again a contradiction because $1\leq a_1<b'_1-1<b_1-1$. Therefore, $a\geq 4$ for all $a$ with $x_a|m_{q,l}$. In particular, $\overline{a-i}=a-i$ for $a\neq t+4$ and $i=1,2,3$.   If $a<k$, then set $m_p:=\x{\e{l}}x_a$ and so we are done. Suppose $a\geq k$ for all $a$ with $x_a|m_{q,l}$. 
 Since $\deg m_{q,l}>1$, there exists $a$ with  $x_a\in\supp(m_{q,l})\cap \supp(\x{\e{q}})$. Suppose $e'_{i_1}=\{a,c\}\in\e{q}$. 
  If $x_{c}|m_{q,l}$, then  $c\neq t+4$ implies that $b_1\in \{a,a-1\}\cap\{c,{c-1}\}$. But  $c\notin \{a,a-1,\overline{a+1}\}$ and hence $ \{a,a-1\}\cap\{c,{c-1}\}=\emptyset$, a contradiction. Thus $x_{c}\centernot|m_{q,l}$ and consequently $x_{c}|m_l$. 
 
Suppose $c=k$.  Then $c\leq a$ and since $\{a,c\}\in E(\bar{C})$ we have $c<a-1$. If $c\neq {a-2}$, then $b\in \{a,{a-1},\overline{a+1}\}$ implies that $e:=\{c, b\}\in E(\bar{C})$. Set $m_p:=(\x{\e{l}}\x{e}\x{e'_1}/(\x{e_1}\x{e_h}))x_a$. We have $m_p\leq m_l$. If  $m_p=m_l$, then $b_1=a$ which implies that $(\x{\e{l}}\x{e}\x{e'_1}/(\x{e_1}\x{e_h}))x_a$ is a smaller presentation for $m_l$, a contradiction. 
 Thus $m_p<m_l$ and $m_{p,l}=x_a$, as desired. Now suppose $k=c={a-2}$. Then $a\geq 5$, and $a_1\in\{a-1,a-2,{a-3}\}$ since otherwise $\{a_1,a-2\}\in E(\bar{C})$ and by interchanging $x_{b_1}$ in $\x{e_1}$ and $x_k=x_{a-2}$ in the presentation of $m_l$ we get a smaller presentation which is  a contradiction.  From $a_1\in\{a-1,a-2,{a-3}\}$, and $a_1+1<b'_1<b_1\in \{a,a-1\}$ we conclude that $a_1={a-3}$, $b'_1=a-1$ and $b_1=a$. Thus  $b\in \{a,a-1,\overline{a+1}\}$ implies that $b=\overline{a+1}$ and therefore interchanging $x_{b'_1}$ in $\x{e_h}$, where $e_h=\{b'_1,b\}=\{a-1,\overline{a+1}\}$,  and $x_k=x_{a-2}$ will give a smaller presentation,    a contradiction. {  Note that since $a\geq 5$, we have $\{b, a-2\}=\{\overline{a+1}, a-2\}\in E(\bar{C})$. }
 
Assume now that $c\neq k$.  Then there is $e_{i_2}\in\e{l}$ with $c\in e_{i_2}$. If $d$ is another vertex of $e_{i_2}$, then we may assume that $d<a$, because $\e{q}\cap\e{l}=\emptyset$ and if $d>a$, then we can set $m_p:=(\x{\e{l}}\x{e'_{i_1}}/\x{e_{i_2}})x_k$ which yields the result. 

First assume  that $b_1=a$. If $b'_1\leq d$, since $a_1+1<b'_1\leq d<a=b_1$, we have   $e:=\{a_1,d\}\in E(\bar{C})$ with $e<e_1$  and hence interchanging $a$ in $e_1$ and $d$ in $e_{i_2}$ will give  a smaller presentation for $m_l$, a contradiction. Thus $b_1=a$ implies that $b'_1>d$ and since $d<b'_1<b_1$ we have  
$e:=\{d, b_1\}\in E(\bar{C})$, because otherwise $d=1$ and $b_1=t+3$ and since $a_1\leq d$ we have $a_1=1$ and $e_1=\{1,t+3\}$, a contradiction. Moreover, $e_{i_2}\neq e_1,e_h$. 
Set $m_p:=(\x{\e{l}}\x{e'_1}\x{e}\x{e'_{i_1}}/(\x{e_1}\x{e_{i_2}}\x{e_h}))x_k$. We have $m_p\leq m_l$.   If $b=a$, then $(\x{\e{l}}\x{e'_1}\x{e}\x{e'_{i_1}}/(\x{e_1}\x{e_{i_2}}\x{e_h}))x_k$ is a smaller presentation for $m_l$, a contradiction. Hence $b\neq a$ and thus $m_p<m_l$ with $m_{p,l}=x_a$.

Now assume that $b_1= a-1$ which implies that $b\in \{a,a-1\}$, because $b\in\{b_1-1, b_1,\overline{b_1+1}\}\cap\{a-1,a,\overline{a+1}\}$. If $d<b'_1$, then $d<b'_1<b_1=a-1\leq b$. If $\{d,b\}= \{1,t+3\}$, then $d=1$ implies that $a_1=1$ and  since $c\neq a-1$ we have $e_1\neq e_{i_2}$ which implies that   $\{1,a-1\}=e_1<e_{i_2}=\{1,c\}$ and hence $b_1=a-1<c$. Moreover, $b=t+3\in \{a-1,a\}$ implies that $a=t+3$ and hence $c=t+3$, a contradiction to $\{a,c\}\in E(\bar{C})$. 
Thus  $\{d,b\}\neq  \{1,t+3\}$ which implies that $e:=\{d,b\}\in E(\bar{C})$.  Set $m_p:=(\x{\e{l}}\x{e}\x{e'_1}\x{e'_{i_1}}/(\x{e_h}\x{e_1}\x{e_{i_2}}))x_k$.
Thus we may suppose that  $b'_1\leq d$. If  $d<a-1$, then set $e:=\{a_1,d\}$ and since $e_1\neq e_{i_2}$ we set $m_p:=(\x{\e{l}}\x{e}\x{e'_{i_1}}/(\x{e_1}\x{e_{i_2}}))x_k$. Suppose now that $d=a-1$.

Note that from $k\leq a$ we conclude that $k\in\{a-1,a\}$ since otherwise if $k=a-2$, then by interchanging $x_{b_1}=x_{a-1}$ in $\x{e_1}$ and $x_{k}=x_{a-2}$ we get a smaller presentation for $m_l$, and 
  if $k<a-2$, setting $e:=\{k,b\}$, we again get $(\x{\e{l}}\x{e'_1}\x{e}/(\x{e_1}\x{e_h}))x_{a-1}$ as a smaller presentation for $m_l$.  


If $x_{a-1}|m_{q,l}$, or if $\deg_{m_q}x_{a-1}< \deg_{m_l}x_{a-1}$, then set $m_{q'}:=(\x{\e{q}}\x{e_{i_2}}/\x{e'_{i_1}})x_{k'}$. Since $m_{q'}<m_q$ and since $\supp(m_{q',l})\subseteq\supp(m_{q,l})$, by induction hypothesis we are done. Suppose $\deg_{m_q}x_{a-1}= \deg_{m_l}x_{a-1}$. Since $x_{a-1}|m_l$ we have $x_{a-1}|m_q$ as well. Note that $k'\neq a-1$, otherwise interchanging $x_{k'}$ and $x_a$ in $\x{e'_{i_1}}$ will result in a smaller presentation for $m_q$, a contradiction. It follows that there exists $e'_{i_3}=\{a-1,f\}\in \e{q}$ for some $f$ with $f\neq c, a_1$ because $\e{q}\cap \e{l}=\emptyset$. 

 If $x_f|m_{q,l}$, then we must have $a-1=b_1\in \{f,\overline{f-1}, \overline{f+1}\}$, a contradiction. 
Thus $x_f\centernot|m_{q,l}$. As $f\neq a,a-1$ we have $f\neq b$ and also $f\neq k$ which implies that  $f$ appears in an edge of $\e{l}$. If $f=b'_1$, then again $\e{q}\cap\e{l}=\emptyset$ implies that $b=a$. Therefore we have $e'_{i_1}=\{a,c\}, e'_{i_3}=\{a-1, b'_1\}\in \e{q}$ and $e_{i_2}=\{a-1,c\}, e_h=\{a,b'_1\}\in \e{l}$ which contradict the fact that $\e{l}$ and $\e{q}$ are the smallest multisets associated to $m_l$ and $m_q$, respectively. Thus $f\neq b'_1$ and hence $f\notin e_1\cup e_h\cup e_{i_2}\cup\{k\}$.  It follows that  there exists $e_{i_4}\neq e_1,e_{i_2}, e_h$  such that $e_{i_4}=\{f,g\}\in \e l$ for some $g$ with $g\neq a-1$. If $g>a$, { then $a\leq t+2$ and hence $\overline{a+1}=a+1$. }
We have $f\neq a+1$ because otherwise we will have  $g>a+2$ and since $k\in\{a,a-1\}$ by interchanging $x_f$ in $\x{e_{i_4}}$ and $x_k$ one gets a smaller presentation for $m_l$ which is a contradiction. {Note that assuming $f=a+1$ one deduces from $g>a$ and $g\notin \{a,a+1,\overline{a+2}\}$ that $a+1<g\leq t+3$ and hence $\overline{a+2}=a+2$.}  
Now set $e:=\{f,a\}\in E(\bar{C})$ and $m_p:=(\x{\e{l}}\x{e}/\x{e_{i_4}})x_k$.  If $g=a$, then we have $e'_{i_1}=\{a,c\}, e'_{i_3}=\{a-1, f\}\in \e{q}$ and $e_{i_2}=\{a-1,c\}, e_{i_4}=\{a,f\}\in \e{l}$ which again contradict the fact that $\e{l}$ and $\e{q}$ are the smallest multisets associated to $m_l$ and $m_q$, respectively. Thus $g\neq a$ which implies that $g<a-1$.  
 If $g\notin \{a_1,a_1+1, \overline{a_1-1}\}$, then interchanging $a-1$ in $e_1$ and $g$ in $e_{i_4}$ will give a smaller presentation for $m_l$, a contradiction. Thus $g\in \{a_1,a_1+1, \overline{a_1-1}\}$. Since $g\geq a_1$ we have $g\in \{a_1, a_1+1\}$. If $g=a_1$, then  $e_1\leq e_{i_4}$  implies that $f\geq a-1$. Since $f\neq a, a-1$ we have $f>a$. This in particular implies that $a\neq t+3$ and hence it follows from $a_1+1<a-1$  that $e:=\{a_1, a\}\in E(\bar{C})$. Now set $m_p:=(\x{\e{l}}\x{e}/\x{e_{i_4}})x_k$. Suppose $g=a_1+1$.  Then  $e:=\{a_1, f\}\in E(\bar{C})$ because  $a_1\leq f$ and $f\notin \{a_1,a_1+1\}$ by $e_{i_4}=\{f, a_1+1\}\in\e{l}$. Moreover, $e':=\{a_1+1, a-1\}\in E(\bar{C})$ because $a_1+1<b'_1<a-1$. If  $f<a-1$, then $(\x{\e{l}}\x{e}\x{e'}/(\x{e_1}\x{e_{i_4}})x_k$ is a smaller presentation for $m_l$ which is a contradiction.  Thus $f>a$ and hence set $m_p:=(\x{\e{l}}\x{e''}/\x{e_{i_4}})x_k$, where $e'':=\{a_1+1, a\}$. 
 This completes the proof.
\end{proof}


Now we extend the ideal $L$ of  Proposition~\ref{step1} to another ideal which contains  $L$ and has linear quotients. 

\begin{thm}\label{linquo of square}
	Let $I\subset S$ be the edge ideal of the graph  $G=\overline{G_{(b)}}$, with  $t\geq 1$. Then the ideal $(I^{s+1}:x_{t+5})$ has linear quotients for  all $s\geq 0$. 
\end{thm}

\begin{proof}
	 Set $J:=  (I^{s+1}:x_{t+5})$. We first determine the minimal generating set $\mathcal{G}(J)$ of $J$. Note that $E(G)=E(\bar{C})\cup\{x_{t+5}x_i:\ 3\leq i\leq t+4\}$, where $C=1-2-\cdots-(t+3)-1$ is the unique induced cycle of $G_{(b)}$ of length$>3$.  Hence, 
	 $$I^{s+1}=\sum_{k=0}^{s+1}I(\bar{C})^{s+1-k}(x_{t+5})^k(x_3,\ldots,x_{t+4})^k.$$
	 By \cite[Proposition~1.2.2]{HHBook}, the ideal $J$ is generated by  monomials $m/\gcd(m, x_{t+5})$, where $m\in I^{s+1}$. It follows that 
	 $$J=I(\bar{C})^{s+1}+\sum_{k=0}^{s}I(\bar{C})^{s-k}(x_{t+5})^k(x_3,\ldots,x_{t+4})^{k+1}.$$
	 Since each edge of $\bar{C}$ contains a vertex in $\{3,\ldots,t+3\}$, we have $I(\bar{C})^{s+1}\subset I(\bar{C})^{s}(x_3,\ldots,x_{t+4})$. Therefore, 
	 $$J=\sum_{k=0}^{s}I(\bar{C})^{s-k}(x_{t+5})^k(x_3,\ldots,x_{t+4})^{k+1}.$$
	 For $0\leq k\leq s$, let $L_{k}:=I(\bar{C})^{s-k}(x_{t+5})^k(x_3,\ldots,x_{t+4})^{k+1}$. Clearly, for $0\leq k, k'\leq s$ with $k\neq k'$ we have $\mathcal{G}(L_k)\cap\mathcal{G}(L_{k'})=\emptyset$, where $\mathcal{G}(L_{k})$ denotes the minimal generating set of $L_k$. Therefore $\mathcal{G}(J)$ is the disjoint union of all $\mathcal{G}(L_k)$ for $0\leq k\leq s$. In particular, $J$ is generated by monomials of degree $2s+1$. For $s=0$, $J$ is generated by variables and hence we have the assertion. Suppose $s\geq 1$.
	 
	 We order the multisets of the edges of $\bar{C}$ as described before Proposition~\ref{step1}.
	Each element $m_l$ of $\mathcal{G}(L_k)$ can be written as  $m_l=\x{\e{l}}{x_{t+5}}^k\x{l}$, where $\e{l}=(e_{1},\ldots,e_{{s-k}})$ is an ordered multiset of the edges of $\bar{C}$ of size $s-k$ with $e_i=\{a_i,b_i\}$, $a_i<b_i$,  and $\x{l}:=x_{j_1}\cdots x_{j_{k+1}}$ with $3\leq j_1\leq \cdots\leq j_{k+1}\leq t+4$.  Similar to the proof of Proposition~\ref{step1} we consider the {\em smallest} presentation for $m_l$, i.e. the  one  in which $\e{l}$ is the smallest possible multiset associated to $m_l$.  So this presentation is unique.

	Now we give an order on the generators of $J$. { To this end we use the lexicographic order $<_{lex}$ on the monomials of the ring $S$ induced by $x_1<x_2<\cdots<x_{t+5}$; see \cite[Section~2.1.2]{HHBook} for the definition of the lexicographic order. 
	}
	
	For $m_q,m_l\in \mathcal{G}(J)$ we let $m_q<m_l$ in the following cases:
	\begin{itemize}
	\item $m_q\in L_{k'}$ and $m_l\in L_{k}$ with $0\leq k'<k\leq s$;
	\item  $m_q,m_l\in L_k$ for some $0\leq k\leq s-1$ and either $\mathbf{e}_q<\mathbf{e}_{l}$ or $\mathbf{e}_q=\mathbf{e}_{l}$ with $\mathbf{x}_q<_{lex}\mathbf{x}_l$;
		\item $m_q,m_l\in L_s$,  and { either $m_q\neq  x_{t+5}^sx_i^{s+1}, m_l\neq x_{t+5}^sx_j^{s+1}$ for all $3\leq i, j\leq t+4$ with $m_q<_{lex}m_l$,  or $m_q= x_{t+5}^sx_i^{s+1}$ and $m_l= x_{t+5}^sx_j^{s+1}$ for some $3\leq i<j\leq t+4$, or  $m_q\neq  x_{t+5}^sx_i^{s+1}$ for all $3\leq i\leq t+4$ and $m_l=  x_{t+5}^sx_j^{s+1}$ for some $3\leq j\leq t+4$.}
			\end{itemize}
	 
	Suppose $\mathcal{G}(J)=\{m_1,\ldots,m_r\}$ with $m_1<\!\cdots\!<m_r$. We show that for any $m_l\in\!\mathcal{G}(J)$ with $l\!>1$, the ideal $\left((m_1,\ldots, m_{l-1}):m_l\right)$ is generated by some variables. Set $J_l:=(m_1,\ldots, m_{l-1})$. By \cite[Proposition~1.2.2]{HHBook}, the ideal $(J_l:m_l)$ is generated by the elements of the  set $\{m_q/\gcd(m_q,m_l):\ 1\leq q\leq l-1\}$. Let $m_{q,l}:=m_q/\gcd(m_q,m_l)$ for $m_q<m_l$. 

	Suppose $m_q=\x{\e q}x_{t+5}^{k'}\x{q}\in \mathcal{G}(L_{k'})$ with $0\leq k'\leq s$ and $\mathbf{e}_{q}=(e'_1,\ldots,e'_{s-k'})\subseteq E(\bar{C})$ with $e'_i=\{a'_i,b'_i\}$, $a'_i<b'_i$, and $\mathbf{x}_{q}=x_{j'_1}\cdots x_{j'_{k'+1}}$ with $3\leq j'_1\leq \cdots\leq j'_{k'+1}\leq t+4$, and suppose  $m_l\in \mathcal{G}(L_{k})$  with $k'\leq k\leq s$. Suppose 	 $\deg m_{q,l}>1$. 
	 We show that there is $1\leq p<l$ such that $\deg m_{p,l}=1$ and $m_{p,l}| m_{q,l}$. This will imply that $J$ has linear quotients. We may assume $k\geq  1$ because by  Proposition~\ref{step1},  $L_0=I(\bar{C})^{s}(x_3,\ldots,x_{t+4})$ has linear quotients.  By the same argument as in the proof of Proposition~\ref{step1} we may assume that $\e{q}\cap \e{l}=\emptyset$.	 
	 First assume $q=1$. Then $m_q=x_1^sx_3^{s+1}$ and $x_1|m_{q,l}$, because $k\geq 1$. 	 {  If   $x_3|\x{l}$, then set $e:=\{1,3\}$  and $m_p:=\x{e}m_l/(x_3x_{t+5})\in L_{k-1}$. Otherwise  we have  $x_3|m_{q,l}$ and  we may set $m_p:=m_lx_3/x_{j_i}$ for some $j_i$.} 
	   Suppose now that $q>1$ and suppose that for all $m_{q'}$ with $q'<q$ there exists  $m_{p'}<m_l$ with $\deg m_{p',l}=1$ and $m_{p',l}|m_{q',l}$. We prove the assertion by induction on $q$.

	 Note that 
	\begin{enumerate}
			\item[(a)] $x_{t+5}\notin \supp(m_{q,l})$, because $\deg_{m_q}x_{t+5}\leq \deg_{m_l}x_{t+5}$. 
		\item[(b)] Assume $a\in \{3,4,\ldots, j_{k+1}-1\}$. Except for the case where  $k=s$ with 
		$m_l= x_{t+5}^sx_a^{s}x_{j_{s+1}}$,  one has $x_a\in (J_l:m_l)$, because $m_p:=x_am_l/x_{j_{k+1}}\in J_l$ and $m_{p,l}=x_a$. 
		\item[(c)] If  $k=s$ and $m_l=x_{t+5}^s x_a^{s}x_{j_{s+1}}$, where $3\leq a<j_{s+1}-1<t+3$, then $e\!:=\{a,j_{s+1}\}\in E(\bar{C})$ and by setting $m_p:=\x{e}m_l/(x_{j_{s+1}}x_{t+5})$ one has $m_p\in L_{s-1}\subseteq J_l$ and $x_a\in (J_l:m_l)$.
				 \item[(d)] For any $a$ with $j_{k+1} + 1 < a < t + 4$, we have $x_a \in (J_l : m_l)$ because
$m_p:=\mathbf{x}_em_l/(x_{t+5}x_{j_{k+1}})\in L_{k-1}\subseteq J_l$.
				 	 \end{enumerate}
		 \medspace

		  If $\supp(m_{q,l}) \cap \{x_3, x_4, \ldots , x_{j_{k+1}-1}\}\neq \emptyset$, then we are
done. Indeed, assume that there exists $x_a \in \supp(m_{q,l})\cap\{x_3, x_4, \ldots , x_{j_{k+1}-1}\}$.
Then by (b) and (c), it is sufficient to check only the case where $k = s$,
$m_l = x^s_{t+5}x^s_ax_{j_{s+1}}$, and $j_{s+1} \in \{a + 1, t + 4\}$.
		  		  Since $\deg_{m_q}x_a\geq s+1$, if $m_q\in L_s$, then  $m_q=x_{t+5}^sx_a^{s+1}>m_l$, a contradiction. Thus $m_q\notin L_s$ and hence there exists $e'_i\in \e{q}$ with  $a\in e'_i$. Suppose $d$ is another vertex of $e'_i$. Since $d\notin\{a,a+1, t+4,t+5\}$ we have $x_d|m_{q,l}$. Set  $m_p:=\x{e'_i}m_l/(x_ax_{t+5})$. Then $m_p<m_l$ and $m_{p,l}=x_d$ and so we are done.  

		
		\medspace
		Moreover, if $x_a\in \supp(m_{q,l})\cap \{x_{j_{k+1}+2}, \ldots, x_{t+3}\}$, then we are again done by (d).	
		Thus, using (a) and the above discussion,  we may suppose that   
		 \begin{eqnarray}\label{pizza}
		\supp(m_{q,l})\subseteq	\{x_1,x_2,x_{j_{k+1}},x_{j_{k+1}+1}, x_{t+4}\}\setminus\{x_{t+5}\}.
						\end{eqnarray}	
			    Note that
			\begin{itemize}
			\item[(i)] if $x_1\in \supp(m_{q,l})$ and  $j_1\neq {t+3}, t+4$,  then set $e:=\{1,j_1\}$ and $m_p:=\x{e}m_l/(x_{j_1}x_{t+5})$;
			\item[(ii)]  if $x_2\in \supp(m_{q,l})$ and there exists $j_i\neq 3,{t+4}$, $1\leq i\leq k+1$, then set $e:=\{2,j_i\}$ and $m_p:=\x{e}m_l/(x_{j_i}x_{t+5})$;
			\item[(iii)]  if $x_{j_{k+1}}\in \supp(m_{q,l})$ and  $j_1<j_{k+1}-1{ <t+3}$, then set $e:=\{j_1,j_{k+1}\}$ and $m_p:=\x{e}m_l/(x_{j_1}x_{t+5})$;
			\item[(iv)] if $x_{j_{k+1}+1}\in \supp(m_{q,l})$ and  $j_1<j_{k+1}{ <t+3}$, then set $e:=\{j_1,j_{k+1}+1\}$ and  $m_p:=\x{e}m_l/(x_{j_1}x_{t+5})$.
			\end{itemize}
			Thus, by (\ref{pizza}) it remains to find $m_p$ in the following cases:
					\begin{itemize}
			\item[(v)] $x_1\in \supp(m_{q,l})$ 
			and  $j_1\in \{{t+3}, t+4\}$; 
			\item[(vi)]  $x_2\in \supp(m_{q,l})$; 
			and  $j_i\in \{3,{t+4}\}$ for all $1\leq i\leq k+1$;
	 { 	\item[(vii)] $x_{j_{k+1}}\in \supp(m_{q,l})$ and  either $j_1\in\{ j_{k+1}-1,j_{k+1}\}$ or $j_{k+1}=t+4$;
		\item[(viii)] $x_{j_{k+1}+1}\in \supp(m_{q,l})$ and either $j_1=j_{k+1}$ or $j_{k+1}=t+3$;
			\item[(ix)] $x_{t+4}\in\supp(m_{q,l})$. }
			\end{itemize} 
			
		{	In (viii) we have $j_{k+1}\neq t+4$ because $x_{t+5}\notin \supp(m_{q,l})$. Since we will check the case $x_{t+4}\in\supp(m_{q,l})$ in (ix), we may suppose in Case~(vii) that $j_{k+1}\neq t+4$ and in Case~(viii)  that $j_{k+1}\neq t+3$.  Moreover, 
		having $j_1\in \{j_{k+1}-1, j_{k+1}\}$ in (vii) we have either $x_{j_1}\in \supp(m_{q,l})$ with  $\supp(\x{l})= \{x_{j_1}\}$ or $x_{j_1+1}\in \supp(m_{q,l})$ with  $\supp(\x{l})= \{x_{j_1}, x_{j_1+1}\}$. In Case~(viii), since $j_1=j_{k+1}$ we get $x_{j_1+1}\in \supp(m_{q,l})$ and $\supp(\x{l})=\{x_{j_1}\}$.  So  combining the two cases (vii) and (viii), we will end up with the following ones:
			
			\begin{itemize}
			\item[(vii')]  
			$x_{j_1}\in \supp(m_{q,l})$   
			and  $\supp(\x{l})=\{x_{j_{1}}\}$;
			\item[(viii')] $x_{j_1+1}\in \supp(m_{q,l})$ 
			and  either $\supp(\x{l})=\{x_{j_{1}}\}$ or $\supp(\x{l})=\{x_{j_{1}}, x_{j_1+1}\}$.
			\end{itemize}}
			
			So we  replace (vii), (viii) with (vii'), (viii'). Now we prove the assertion in the above five cases. Note that since $m_q\neq m_l$ there exists $1\leq b\leq t+5$ such that $\deg_{m_q}x_b<\deg_{m_l}x_b$. Suppose $B$ is the set of all such $b$.

			\medspace
			Case~(v): Since $j_1\in \{t+3,t+4\}$ we have $\supp(\x{l})\subseteq \{x_{t+3},x_{t+4}\}$ and hence $j_{k+1}\in \{t+3,t+4\}$ implies that  $\supp(m_{q,l})\subseteq\{x_1,x_2,x_{t+3}{, x_{t+4}}\}$ by (\ref{pizza}). Since $x_1\in \supp(m_{q,l})$, there exists $e'_i\in\e{q}$ with $e'_i=\{1,b'_i\}$ for some $3\leq b'_i\leq t+2$. Since $x_{b'_i}\notin\{x_1,x_2,x_{t+3},x_{t+4}\}$ we have $x_{b'_i}|m_l$, and it follows from $\supp(\x{l})\subseteq \{x_{t+3},x_{t+4}\}$ that $x_{b'_i}\centernot|\x{ l}$. 
			Thus there exists  $e_j\in\e{l}$ with $b'_i\in e_j$. If $d$ is another vertex of $e_j$, then $d> 1$ because $\e{l}\cap\e{q}=\emptyset$. Set $m_p:=\x{e'_i}m_l/\x{e_j}$ and so we are done in this case.  This case together with (i) imply that if $x_1|m_{q,l}$, then we have the desired $m_p$. Suppose  in the remaining cases that $x_1\notin\supp(m_{q,l})$. 
					
			\medspace
			Case~(vi):  Since $x_{j_{k+1}}\in \supp(\x{l})\subseteq \{x_3,x_{t+4}\}$ we have $\supp(m_{q,l})\subseteq\{x_2,x_{3},x_4{, x_{t+4}}\}$ by (\ref{pizza}). 
			Since $x_2\in \supp(m_{q,l})$, there exists $e'_i\in\e{q}$ with $e'_{i}=\{2,b'_{i}\}$ for some $4\leq b'_{i}\leq t+3$. 
			
			First suppose $m_l\in L_s$. Then $m_l=x_{t+5}^s\x{l}$ and $x_{b'_{i}}\centernot|m_l$ because $\supp(\x{l})\subseteq \{x_3,x_{t+4}\}$. Thus $x_{b'_{i}}|m_{q,l}$ and therefore $b'_{i}=4$. In case $x_{t+4}|\x l$ we set $m_p:=x_{4}m_l/x_{t+4}$. Otherwise, we have $m_l=x_{t+5}^sx_3^{s+1}$, and hence we can set $m_p:=x_{t+5}^sx_3^{s}x_{4}$. 
			
			Suppose now that $m_l\notin L_s$. 		
			There exists  $e_{j}=\{a_{j},b_{j}\}\in\e{l}$ with $a_{j}\neq 2$. If $a_{j}\neq 1$ then set $e:=\{2,b_{j}\}$ and $m_p:=\x{e}m_l/\x{e_{j}}$. 			Suppose that  $e_{j}=\{1,b_{j}\}$ for all $e_{j}\in\e{l}$ with $2\notin e_{j}$.   
			
			If $x_{b'_{i}}|m_{q,l}$, then $b'_{i}=4$. 
			If $x_{t+4}|\x{l}$, then set $m_p:=x_{4}m_l/x_{t+4}$. Otherwise, we have  $\supp(\x{l})=\{x_3\}$. Then $b_{j}=3$ for all $e_j = \{1, b_j\} \in \e l$, since otherwise we get a contradiction to the fact that we have considered the smallest presentation for $m_l$.		
			If $x_2|m_l$, then there exists $e_{r}=\{2,b_{r}\}\in \e{l}$, where $b_{r}>4$ because $\e{q}\cap\e{l}=\emptyset$. Then set $m_p:=\x{e'_{i}}m_l/\x{e_{r}}$. If $x_2\centernot|m_l$, then $m_l=x_{t+5}^kx_1^{s-k}x_3^{s+1}$. Thus $3\in B$ because $3\notin e'_i = \{2, 4\}\in \e q$. If $1\in B$, then set $e:=\{1,4\}$ and $m_{q'}=\x{e}m_q/\x{e'_i}$, and if  $1\notin B$, then there exists $e'_f=\{1,b'_f\}\in \e{q}$ with $b'_f>3$ because $e_j=\{1, 3\} \in \e l$ and $\e q \cap \e l = \emptyset$, so set $m_{q'}=\x{e_j}m_q/\x{e'_f}$ and use induction. 
			
			Now suppose $x_{b'_i}\centernot|m_{q,l}$. Since $\supp(\x{l})\subseteq\{x_3,x_{t+4}\}$, we have $e_r:=\{1,b'_i\}\in \e{l}$ because all  edges in $\e{l}$ contain either $1$ or $2$ and  $\e{q}\cap\e{l}=\emptyset$.  Since $b'_i>3$ we have $\supp(\x{l})=\{x_{t+4}\}$ because otherwise we get a smaller presentation for $m_l$. 			
			If $1\notin B$, then $e'_f:=\{1,b'_f\}\in \e{q}$ with $b'_f\neq b'_i$ because $\e{q}\cap\e{l}=\emptyset$. If $x_{b'_f}|m_{q,l}$, then set $m_p:=x_{b'_f}m_l/x_{t+4}$. If $x_{b'_f}\centernot|m_{q,l}$, then we have $\{2,b'_f\}\in \e{l}$ since $\e{q}\cap\e{l}=\emptyset$ and each edge of $m_l$ contains either $1$ or $2$. Now we have $\{1,b'_i\}, \{2,b'_f\}\in \e{l}$ and $\{1,b'_f\}, \{2,b'_i\}\in \e{q}$ which contradict the fact that both $\e{q}, \e{l}$ are the smallest multisets associated to $m_q,m_l$ respectively. Suppose $1\in B$. Then we can use inductive hypothesis for $m_{q'}:=\x{e_r}m_q/\x{e'_i}$.
			 So we are done in this case too. By settling this case and according to Case~(ii) we have the desired $m_p$ if $x_2|m_{q,l}$. Suppose in the remaining cases that  $\supp(m_{q,l})\subseteq\{x_{j_{k+1}}, x_{{j_{k+1}}+1}{ , x_{t+4}}\}$.

\medspace
Case~(vii') 
 Since $\supp(\x l)=\{x_{j_1}\}$, we have   $\supp(m_{q,l})\subseteq\{x_{j_1}, x_{j_1+1}, { x_{t+4}\}}$.   
There exists $e'_{i_1}=\{j_1,c\}\in\e{q}$ for some $c$ because otherwise $\deg_{m_q}x_{j_1}\leq k'+1\leq k+1=\deg_{m_l}x_{j_1}$, a contradiction. It follows that $j_1\leq t+3$. 
Since $x_c\centernot|m_{q,l}$ we have $x_c|m_l$ and since $c\neq j_1$, there exists $e_{i_2}=\{c,d\}\in \e{l}$ for some $d$. By $\e{q}\cap\e{l}=\emptyset$, we have $d\neq j_1$. Note that $d<j_1$, because otherwise interchanging $x_d$ in $\x{e_{j}}$ and $x_{j_1}$ in $\x{l}$ will result in a smaller presentation of $m_l$.

Suppose  $d\!<\!j_1\!-\!1$. { If $\!\{d,\!j_1\}\!\neq\!\{1,\!t+3\}$,} then set $e\!:=\!\{d,\!j_1\}$ and $m_p\!:=\!\x{e}\x{e'_{i_1}}\!m_l/(\x{e_{i_2}}\!x_{j_1}\!x_{t+5})$. 
{  If $\{d, j_1\}= \{1,t+3\}$, then in case  $1\in B$, set $m_{q'}:=\x{e_{i_2}}m_q/\x{e'_{i_1}}<m_q$ and use induction.  In case $1\notin B$, there exists $e'_{i_3}=\{1,f\}\in \e{q}$ for some $f\notin\{ c, 1,2,t+3, t+4\}$ which implies that $x_f\centernot| m_{q,l}$. Hence $x_f|m_l$ and thus $e_{i_4}=\{f,g\}\in \e{l}$ for some $g\neq 1$.  If $g=t+3$, then $\{c,t+3\}, \{1,f\}\in \e q$ and $\{1,c\}, \{f,t+3\}\in \e l$  which contradict the fact that $\e{q}$ and $\e{l}$ both have minimum presentations. Thus $g\neq t+3$. If $g\neq t+2$, set $e:=\{g,t+3\}$ and $m_p:=\x{e}\x{e'_{i_1}}\x{e'_{i_3}}m_l/(\x{e_{i_2}}\x{e_{i_4}}x_{t+3}x_{t+5})$, and if $g=t+2$, set $e=\{1,t+2\}, e'=\{f, t+3\}$ and $m_p:=\x{e}\x{e'}\x{e'_{i_1}}m_l/(\x{e_{i_2}}\x{e_{i_4}}x_{t+3}x_{t+5})$.

}
Suppose $d=j_1-1$. If $j_1-1\in B$, then set $m_{q'}:=\x{e_{i_2}}m_q/\x{e'_{i_1}}$ and use induction hypothesis. If $j_1-1\notin B$, then $x_{j_1-1}|m_q$. If $x_{j_1-1}|\x{q}$, then interchanging $x_{j_1-1}$ in $\x{q}$ and $x_{j_1}$ in $\x{e'_{i_1}}$ will give a smaller presentation of $m_q$, a contradiction. Therefore there exists $e'_{i_3}=\{j_1-1, f\}\in \e{q}$ for some $f\neq c$.  Then $f\notin \{j_1-2,j_1-1, j_1\}$. If $x_f|m_{q,l}$, then $f={j_1+1}$. Set $m_p:=\x{e'_{i_1}}\x{e'_{i_3}}m_l/(\x{e_{i_2}}x_{j_1}x_{t+5})$. Suppose $x_f\centernot|m_{q,l}$. 
 It follows that there exists $e_{i_4}=\{f, g\}\in \e{l}$ for some $g\neq j_1-1$. We have $f\neq {j_1+1}$ because otherwise, it follows that $g\notin \{j_1,j_1-1,{j_1+1}\}$ which implies that  $m_l$ will have a smaller presentation by interchanging $x_f$ in $\x{e_{i_4}}$ and $x_{j_1}$ in $\x{l}$. Since $f\neq j_1+1$ we have either $g< j_1-1$ or $g=j_1$ because otherwise one can interchange $x_g$ in $\x{e_{i_4}}$ and $x_{j_1}$ in $\x{l}$ to get a smaller presentation. If $g=j_1$, then we have  $\{j_1,c\}, \{j_1-1, f\}\in \e{q}$ and $\{j_1-1, c\}, \{j_1,f\}\in \e{l}$ which contradict the fact that both $\e{q}, \e{l}$ are the smallest multisets associated to $m_q, m_l$, respectively.  Thus $g<j_1-1$. { In case $\{g, j_1\}\neq \{1,t+3\}$, set $e:=\{g, j_1\}$ and $m_p:=\x{e'_{i_1}}\x{e'_{i_3}}\x{e}m_l/(\x{e_{i_2}}\x{e_{i_4}}x_{j_1}x_{t+5})$. In case $\{g, j_1\}= \{1,t+3\}$, set $e:=\{1,t+2\}, e':=\{f,t+3\}$ and $m_p:=\x{e'_{i_1}}\x{e}\x{e'}m_l/(\x{e_{i_2}}\x{e_{i_4}}x_{t+3}x_{t+5})$. Note that since $e'_{i_3}, e_{i_4}\in E(\bar{C})$ we have $f\notin \{1,t+2,t+3\}$ and hence $e'\in E(\bar{C})$.  Thus we are done in this case too. 
  }
 
 \medspace 
{  In general, if $x_{j_1}\in \supp(m_{q,l})$, then by (\ref{pizza}) we have ${j_1}\in \{j_{k+1}, j_{k+1}+1, t+4\}$.  But ${j_1}\leq j_{k+1}$  implies that  $j_1=j_{k+1}$ and hence 
 $\supp{(\x l)}=\{x_{j_1}\}$. Thus by the  discussion in (vii') we are done if $x_{j_1}\in \supp(m_{q,l})$. Therefore, we may assume in the rest of the proof that 
 $x_{j_1}\notin\supp(m_{q,l})$.}

\medspace
Case~(viii'):  
Since $\supp(\x l)\subseteq\{x_{j_1},x_{j_1+1}\}$,  we have $\supp(m_{q,l})\subseteq\{x_{j_1+1},x_{j_1+2},x_{t+4}\}$ by (\ref{pizza}). Moreover, if  $j_1+1= t+4$, then $\supp(m_{q,l})=\{x_{t+4}\}$ and since this case will be discussed in (ix) we may assume here that $j_1+1\neq t+4$. 
Assume first that  $x_{j_1}|\x{\e q}$. Then $e'_{i_1}=\{j_1,c\}\in\e{q}$ for some $c$. Since $c\notin\{j_1, j_1+1\}$, we have $e_{i_2}=\{c,d\}\in\e l$ for some $d$. Since $\e{q}\cap \e l=\emptyset$ and since we have the smallest presentation of $m_l$ we have $d<j_1$. {  If $\{d, j_1+1\}\neq \{1,t+3\}$, then }
set $e:=\{d,j_1+1\}$ and $m_p:=\x{e'_{i_1}}\x{e}m_l/(\x{e_{i_2}}x_{j_1}x_{t+5})$. {  If $\{d, j_1+1\}=\{1,t+3\}$, then set $e:=\{c, t+3\}, e':=\{1,t+2\}$ and $m_p:=\x{e}\x{e'}m_l/(\x{e_{i_2}}x_{t+2}x_{t+5})$. Note that $e:=\{c, t+3\}\in E(\bar{C})$ because $c\notin \{1,t+2,t+3\}$.

Now assume $x_{j_1}\centernot|\x{\e q}$. Suppose $x_{j_1+1}\centernot|\x{\e q}$.  Since $x_{j_1+1}|m_{q,l}$ we have $x_{j_1+1}|\x q$. If $x_{j_1+2}|\x{\e q}$, then $e:=\{a,j_1+2\}\in \e q$ for some $a\notin\{j_1+1, j_1+2, \overline{j_1+3}\}$. Since $x_{j_1}\centernot|\x{\e q}$ we have $a\neq j_1$ too. Thus $\{a, j_1+1\}\in E(\bar{C})$ and hence one can interchange $x_{j_1+2}$ in $\x e$ and $x_{j_1+1}$ in $\x q$ to get a smaller presentation for $m_q$, a contradiction. Thus $x_{j_1+2}\centernot|\x{\e q}$.  Therefore  $\supp(\x{\e q})\cap \supp(m_{q,l})=\emptyset$ which implies that 
$\x{\e q}|\x{\e l}$ and since $k'\leq k$ we have $\e q=\e l$.   But $\e{q}\cap \e l=\emptyset$ implies that $\e q=\emptyset =\e l$. Thus $m_q, m_l\in L_s$.  If $m_l=x_{t+5}^sx_{j_1}^{s+1}$, then set $m_p:=x_{t+5}^sx_{j_1}^{s}x_{j_1+1}$. Suppose $m_l= x_{t+5}^sx_{j_1}^{r}x_{j_1+1}^{s+1-r}$, where $0<r<s+1$. By the order of the generators of $L_s$ we have 
$m_q\neq x_{t+5}^sx_{j_1+1}^{s+1}$.   Since $\supp(m_q)\subseteq \supp(m_{q,l})\cup \supp(m_l)\subseteq \{x_{j_1}, x_{j_1+1},x_{j_1+2},x_{t+4}, x_{t+5}\}$ and $m_q<m_l$, we have $x_{j_1}\in \supp(m_q)$.  If $x_{j_1+2}\in \supp(m_q)$ ($x_{t+4}\in \supp(m_q)$ resp.), then set $m_{q'}:=x_{j_1+1}m_q/x_{j_1+2}$ ($m_{q'}:=x_{j_1+1}m_q/x_{t+4}$ resp.) and use induction. Otherwise, we have $m_q=x_{t+5}^sx_{j_1}^{r'}x_{j_1+1}^{s+1-r'}$ with $0<r'<r$ because $m_q<m_l$.  Set $m_{q'}:=x_{j_1}m_q/x_{j_1+1}$ and use induction. 
}
Suppose now that  $x_{j_1+1}|\x{\e q}$. There exists $e'_{i_1}=\{j_1+1, c\}\in\e{q}$ for some $c$. Since $x_c\notin \supp(\x l)\cup\supp(m_{q,l})$, there exists $e_{i_2}=\{c,d\}\in \e{l}$ for some $d$ with $d\neq j_1+1$. If $d>j_1+1$, then set $m_p:=\x{e'_{i_1}}m_l/\x{e_{i_2}}$. If $d<j_1-1$, then set $e:=\{d,j_1\}$ which is an edge of $\bar{C}$ because $j_1\neq t+3$. 
Now set $m_p:=\x{e}\x{e'_{i_1}}m_l/(\x{e_{i_2}}x_{j_1}x_{t+5})$.   If $d=j_1-1$, then $c\neq j_1-1$ and hence one can  set $e:=\{j_1-1,j_1+1\}, e':=\{c,j_1\}$ which are edges of $\bar{C}$. Now set $m_p:=\x{e}\x{e'}m_l/(\x{e_{i_2}}x_{j_1}x_{t+5})$. Suppose $d=j_1$ which implies that  $c\neq j_1-1$.  It follows that  $x_{j_1}\centernot|\x{q}$ because otherwise, interchanging $x_{j_1}$ in $\x{q}$ and $x_{j_1+1}$ in $\x{e'_{i_1}}$ will give a smaller presentation for $m_q$.  Since $x_{j_1}\centernot|\x{\e q}$ we have 
${j_1}\in B$. Set $m_{q'}=\x{e_{i_2}}m_q/\x{e'_{i_1}}$ and use induction. So we are done also in this case.

{ Now by (iii), (iv), (vii') and (viii') we may assume in the remaining case that $\supp(m_{q,l})=\{x_{t+4}\}$.}

\medspace
 Case~(ix): If  there exists $b\in B\setminus\{1,2, t+5\}$, then setting $m_{q'}:=x_bm_q/x_{t+4}$, we are done by induction hypothesis. 
Suppose $B\subseteq \{1,2,t+5\}$.  
	 
	If $1\in B$, then $m_q, m_l\notin L_s$ and there exists $e'_{i}=\{a'_{i},b'_{i}\}\in\e{q}$ with $1\notin e'_{i}$. If $a'_{i}\neq 2$ set $e:=\{1,a'_{i}\}$,  else if $b'_{i}\neq t+3$ set $e:=\{1,b'_{i}\}$. Then $m_{q'}:=\x{e}m_q/\x{e'_{i}}<m_q$ and $m_{q',l}$ divides $m_{q,l}$. By induction hypothesis we are done. Suppose for all $e'_{i}\in\e{q}$ with $1\notin e'_{i}$ one has $e'_{i}=\{2,t+3\}$. Since  $\deg_{m_q}{x_1}<\deg_{m_l}{x_1}$ there exists  $e_j=\{1,b_{j}\}\in\e{l}$ with $b_{j}\neq 1,2,t+3$ and  since $\e{q}\cap\e{l}=\emptyset$  we have $\deg_{\x{e_q}}{x_{b_{j}}}=0$ for all such $b_{j}$. Since $b_{j}\notin B$, we must have $x_{b_{j}}|\x{q}$. 			 If $b_{j}\neq 3$ by interchanging $x_{b_{j}}$ in $\x{q}$ and $x_{t+3}$ in $\x{e'_{i}}$ we get a smaller presentation of $m_q$ which is a contradiction. Thus $b_{j}=3$. Set $m_{q'}:=\x{e_j}x_{t+3}m_q/(\x{e'_{i}}x_3)$. Then $m_{q'}<m_q$ and $m_{q',l}|m_{q,l}$ and so we are done by induction hypothesis. 			 
			
			Now assume $1\notin B$ and $2\in B$. Again $m_q, m_l\notin L_s$ and there exists $e'_{i}=\{a'_{i},b'_{i}\}\in\e{q}$ with $2\neq a'_{i}<b'_{i}$. If $e'_{i}=\{1,b'_{i}\}$ for all $e'_i\in \e{q}$ with $2\notin e'_{i}$, then $x_1|m_{q,l}$ because otherwise $s-k'=\deg_{m_q}x_1+\deg_{m_q}x_2<\deg_{m_l}x_1+\deg_{m_l}x_2\leq s-k$ and hence $k'>k$, a contradiction.  But $x_1\in \supp(m_{q,l})=\{x_{t+4}\}$  is also a contradiction. Therefore  there exists $e'_{i}=\{a'_{i},b'_{i}\}\in\e{q}$ with $2< a'_{i}<b'_{i}$. 
			 Set $e:=\{2,b'_{i}\}$ and $m_{q'}:=\x{e}m_q/\x{e'_{i}}$. Then $m_{q'}<m_q$ and $m_{q',l}|m_{q,l}$ and so we are again done by induction hypothesis. 		
			
			{ Suppose now that 
$B=\{t+5\}$.  Since $\supp(m_{q,l})=\{x_{t+4}\}$ we have  
			\begin{eqnarray}\label{simple}
			\x{\e{q}}(\x{q}/m_{q,l})=\x{\e{l}}\x{l}. 
			\end{eqnarray}
			Since $\deg_{m_q}x_{t+5}<\deg_{m_l}x_{t+5}$, we have $k'<k$. Moreover, $\deg_{m_l}x_1+\deg_{m_l}x_2\leq s-k$  which implies by (\ref{simple}) that 
			at most $s-k$ edges of $\e q$ contain either $1$ or $2$.   
Now we choose $s-k+1$ edges $e''_1,\ldots, e''_{s-k+1} \in \e{q}$  with the property that no edge in $\e{q}\setminus \{e''_1,\ldots, e''_{s-k+1}\}$ contains $1$ or $2$.  Set $\e{p}:=\{{e''_1}, \ldots, e''_{s-k+1}\}$  and $\x{p}:=x_{t+4}\x{\e{q}}\x{q}/(\x{\e p}m_{q,l})$.  It follows from  the choice of $\e p$ that neither $x_1$ nor $x_2$  divides $\x p$. 
 Hence $m_p:=\x{\e{p}}x_{t+5}^{k-1}\x{p}\in L_{k-1}$. Since by (\ref{simple}), we have $\x{\e p}\x p/x_{t+4}=\x{\e{l}}\x{l}$, we conclude that $m_{p,l}=x_{t+4}$.  	
This completes the proof.}
\end{proof}


Now we  use Theorem~\ref{linquo of square} to show that $I(\overline{G_{(b)}})^k$ has a linear resolution  for $s\geq 2$ when $G_{(b)}$ does not have an induced  $4$-cycle,  that is the number of its vertices is more than or equal to $7$.
\begin{thm}\label{I^k has lin res}
	Let $G$ be a   graph on $n\geq 7$ vertices such that  $G_{(b)}$ is its complement. Let $I:=I(G)$ be the edge ideal of $G$. Then $I^s$ has a linear resolution for $s\geq 2$.
\end{thm}

\begin{proof} 
	By construction, $n=t+5$, where $t\geq 2$. We apply Lemma~\ref{Dao} for $I^s$ and $x:=x_{t+5}$ to prove the assertion. To  this end, we first compute $\mathrm{reg}(I^s+( x_{t+5}))$.  Setting $C=1-2-\cdots-(t+3)-1$, 	for all $s\geq 1$ we have 
	\begin{align*}
	I^s+( x_{t+5})&=(I(\bar{C})+(x_{t+5})(x_3,\ldots, x_{t+4}))^s+(x_{t+5})
	=I(\bar{C})^s+(x_{t+5}).
	\end{align*}
	 Since $x_{t+5}$ does not appear in the support of the generators of $I(\bar{C})^s$, we have 
	$$\mathrm{reg}(I^s+(x_{t+5}))=\mathrm{reg}(I(\bar{C})^s+(x_{t+5}))=\mathrm{reg}(I(\bar{C})^s).$$
	It is proved in \cite[Corollary~4.4]{BHZ} that $I(\bar{C})^s$ has a linear resolution for $s\geq 2$ when $|C|>4$, which is the case here because $t+3>4$. Thus $\mathrm{reg}(I^s+(x_{t+5}))= 2s$ for $s\geq 2$. 
	 On the other hand $(I^s:x_{t+5})$ has linear quotients  by  Theorem~\ref{linquo of square}, and it is seen in its proof  that   $(I^s:x_{t+5})$  is generated in degree $2s-1$ for $s\geq 1$. Therefore, $(I^s:x_{t+5})$ has a $(2s-1)$-linear resolution for $s\geq 1$, see \cite[Theorem~8.2.1]{HHBook}, and hence $\mathrm{reg}((I^s:x_{t+5}))=2s-1$. Now using Lemma~\ref{Dao} we have $\mathrm{reg}(I^s)\leq 2s$ for $s\geq 2$. Since $I^s$ is generated in degree $2s$ we conclude that $I^s$ has a linear resolution for $s\geq 2$.
\end{proof}

\end{document}